\documentclass{article}
\usepackage{amsmath}
\usepackage{amsfonts}
\usepackage{amssymb}
\usepackage{graphicx}
\usepackage{tikz,pgfplots}
\usetikzlibrary{shapes,snakes}
\usepackage{epstopdf}
\usepackage{color}%
\usepackage{subfigure}
\usepackage{enumerate}
\usepackage[hmargin=3cm,vmargin=3cm]{geometry}
\providecommand{\U}[1]{\protect\rule{.1in}{.1in}}
\providecommand{\U}[1]{\protect\rule{.1in}{.1in}}
\newtheorem{theorem}{Theorem}

\newtheorem{example}{Example}
\newtheorem{lemma}[theorem]{Lemma}

\newtheorem{remark}[theorem]{Remark}

\newenvironment{proof}[1][Proof]{\noindent\textbf{#1.} }{\ \rule{0.5em}{0.5em}}

\DeclareMathOperator\erf{erf}
\begin{document}

\title{$ hp $-version collocation method for a class of nonlinear Volterra integral equations of the first kind}
\author{Khadijeh Nedaiasl\thanks{Institute for Advanced Studies in Basic Sciences,
Zanjan, Iran, e-mail: \texttt{nedaiasl@iasbs.ac.ir  \&   knedaiasl85@gmail.com.} }
\and Raziyeh Dehbozorgi\thanks{School of Mathematics, Iran University of Science and Technology, Tehran 16844,   Iran, e-mail:
\texttt{r.dehbozorgi2012@gmail.com.} }
\and Khosrow Maleknejad \thanks{School of Mathematics, Iran University of Science and Technology, Tehran 16844,   Iran,  e-mail:\texttt{maleknejad@iust.ac.ir.}}}
\maketitle
\begin{abstract}
In this paper, we present a collocation method for nonlinear Volterra integral equation of the first kind. This method benefits from the idea of  $hp$-version projection methods. We provide an approximation based on the Legendre polynomial interpolation. The convergence of the proposed  method is completely studied and an error estimate under the $L^2$-norm is provided.  
Finally, several numerical experiments are presented in order to verify the obtained theoretical results.
\end{abstract}
\vspace{1em} \noindent\textbf{Keywords:} {nonlinear operator, first kind Volterra integral equation, $hp$-collocation method, error analysis.  }

\vspace{1em} \noindent\textbf{2010 Mathematics Subject Classification: 47H30, 45D05, 65L60, 65L70. }

\section{Introduction}
Volterra integral equations of first kind are important in theory and  application, for example solving an exterior homogeneous wave equation with Dirichlet boundary condition leads to a time dependent single layer boundary integral equation which can be seen as a Volterra integral equation of the first kind \cite{MR2875249}. They could be categorized according to their kernels in two types. The first equations are those with well-behaved kernels and the second types have unbounded kernels at $ s=t$, like Abel's integral equation \cite{linz}. 

This paper deals with the numerical solution of nonlinear Volterra integral equation of first kind
\begin{equation}\label{asli}
\mathcal{K}u(t)	:=\int_{0}^{t} \kappa(s,t) \psi(s,u(s))\mathrm{d}s=f({t}), \quad  0\leq t\leq T < \infty.
\end{equation}
A general form of Eq. \eqref{asli} can be expressed as follows:
\begin{equation}\mathcal{K}u(t):=\int_{0}^{t} \kappa(s,t,u(s))\mathrm{d}s=f(t), \quad  t \in I:=[0,T].\end{equation}

 A lot of theoretical and numerical researches  have been devoted to the second kind Volterra integral equations. The theoretical study of them  is given in \cite{MR1050319} and a comprehensive numerical investigation based on collocation method is presented  in  \cite{MR2128285}.
 But a few numbers of them deal with the numerical solution of the first kind integral equations, especially the nonlinear integral equations. 

Eggermont has studied the numerical solution of the Volterra integral equation of the first kind in a series of papers \cite{egger1, egger2, egger3}. For the linear case, the collocation scheme has been analyzed as a \textit{projection method} using piecewise polynomials \cite{egger1}. In addition, the super-convergence property of the collocation projection method  is studied  in \cite{egger3}. The quadrature method for nonlinear Volterra integral equation has been studied in \cite{egger2} by considering  it as a collocation-projection method. Furthermore, an asymptotic optimal error estimation and a comprehensive study of the zero stability of the method are presented. 

Brunner et al. have given a comprehensive convergence analysis for collocation, (quadrature) discontinuous  Galerkin and full discontinuous Galerkin methods for linear Volterra integral equation of the first kind extensively  \cite{brunner1977, kauthenbrunner, MR2993111}. Furthermore, the global order of convergence for the collocation method in the space of piecewise polynomials of degree $ m\geq 0 $, with jump discontinuities on the set of knots is studied in \cite{MR0483586, MR619921}. 
Recent studies for the linear Abel's integral equations based on 
finite element Galerkin method are analyzed in \cite{eggermontabel, Vögeli2018}. 
In mentioned works, to obtain an efficient approximation, one should increase the number of  mesh points ($ h$-version) or the degree of polynomials in the expansion ($ p$-version). 

In order to have the advantages of both $ h $- and $ p $-versions, it is possible to vary time steps and approximation orders simultaneously  which is called $ hp$-version methods. Kauthen and Brunner by following similar ideas have studied the convergence analysis of two-step collocation method based on a Runge-Kutta approach for the first kind Volterra integral equation \cite{kauthenbrunner}. A multi-step collocation method for the second kind Volterra integral equations and its  linear stability properties   have been studied by Conte and Paternoster \cite{CONTE2009}. Recently, Zhang and Liang \cite{liang2018} have modified these approaches and introduced a type of multi-step collocation method by using the Lagrange polynomials in each subinterval for the first kind linear Volterra equation. In such methods, the resulting system  can be solved efficiently, in addition the flexibility of method makes it more suitable for large $ T $ \cite{MR3434874}. { In almost all of the mentioned works,  the linear Volterra integral equations have been studied and  interesting techniques are proposed for solving them.  In spite of the abundant research for the linear cases of Eq. \eqref{asli} in the literature, few approaches deal with the nonlinear ones.}

Due to the efficiency and accuracy, the $ hp$-version Galerkin and collocation methods have received considerable attentions. For example, the $ hp $-version of discontinuous Galerkin and Petrov-Galerkin have been studied for integro-differential equations of Volterra types, for more details see \cite{mustapha, MR3266490}.
Sheng et al. have introduced a multi-step Legendre-Gauss spectral collocation method and given a full analysis of convergence in $ L^2$-norm  for the nonlinear Volterra integral equations of the second kind \cite{sheng}. This approach has been extended to the Volterra integral and integro-differential equations with vanishing delays \cite{MR3434874, MR3488071}. Locally varying time steps makes these methods popular for investigating the numerical solution of integral equations with weakly singular kernels, this idea is completely studied  in \cite{MR3673690}.

 The  aim of this paper is to analyze 
the $ hp$-version Legendre collocation method for the first kind nonlinear Volterra integral equations. An important aspect of this method is its flexibility with respect to the step-size and the order of polynomials in each sub-interval. As we will see in the numerical experiments, the proposed collocation method works for the long-time integration intervals and  also it gives remarkable results for the approximation of the equations with non-smooth solutions.
The notations used in this paper are borrowed from \cite{guo, sheng}.
	 	
This paper is organized in the following way. In Section \ref{sec 2} we give some regularity results for the nonlinear Volterra integral equation of the first kind. Section \ref{prelim} is devoted to the description of the $ hp$-version collocation method for the first kind nonlinear Volterra integral equation. In Section \ref{error analysis}, an error analysis of the proposed method is provided in $L^2$ spaces. 
Finally, in order to show the applicability and efficiency of the method and compare with other methods, several examples with smooth and non-smooth solutions are illustrated in Section \ref{numerical experiments}.
	 
\section{Well-posedness of the  problem} \label{sec 2}In order to introduce a numerical scheme for the solution of the Eq. \eqref{asli}, knowledge of the smoothness properties of the exact solution is
necessary. The existence of the solution for this kind of equation is investigated by  Banach and Schauder's fixed point theorems, for more detail see \cite{MR2128285, hack} and references therein.
In the terminology of the articles 
\cite{ MR0329290, MR1350718, MR847018,  egger2}, the
regularity properties of the solution have been established under some assumptions on the right hand side function  $ f(t) $, the kernel $ \kappa(s,t) $ and $ \psi(s,u) $ or in more general case on $ \kappa(s,t,u) $ where a unique continuous solution has been obtained.
Those assumptions { show} that the first-kind Volterra integral equation  is converted into a second-kind Volterra integral equation which is well-posed in the case of having smooth kernel and right hand side function \cite{MR2128285}. Here, we attempt to reduce the smoothness conditions and investigate the existence of  solution in the Hilbert spaces. Let  $H^{m}(I)$ denote the space of Lebesgue measurable functions with $m$-th  derivatives  in $L^{2}(I)$.  

\begin{theorem}
	Assume that the Eq. \eqref{asli} satisfies the following assumptions
	\begin{enumerate}[i.]
	\item $f(t)\in H^{^m}(I), ~f(0)=0,$ \label{item1}
	\item $\kappa(s,t) \in C^{^{m}}(I\times I)$ and $ \kappa(t,t)\neq 0$ $\text{for all } t \in I$, \label{item2}
    \item $\psi(s,u)\in H^{m-1}(I\times \mathbb{R})$,\label{item3}
    \item $\inf \Big\{  \vert \frac{\partial \psi }{\partial u}(s,u) \vert \ \big| (s,u) \in I \times \mathbb{R}\Big\} \geq M>0,$ \label{item4}
    \item $\psi (s,u)$ is Lipschitz continuous w.r. to $ u$. \label{item5}
    \end{enumerate}
	Then it has a unique solution $ u \in H^{^{m-1}}(I)$.
\end{theorem}
	\begin{proof}
		At the first step, we differentiate both sides of  Eq. \eqref{asli} in the sense of distributional derivative, hence from the assumptions (\textit{\ref{item1}}-\textit{\ref{item3}}), it reads that 
		\begin{equation}\label{sec}
		\psi(t,u(t))+\int_0^t \frac{\kappa_t(s,t)}{\kappa(t,t)}\psi(s,u(s))\mathrm{d}s=\frac{f^\prime(t)}{\kappa(t,t)}.
		\end{equation}     
	These conditions lead that each function $ u(t)$ is a solution of Eq. \eqref{sec} if and only if it is a solution of Eq. \eqref{asli}.
In order to prove the existence of a solution for Eq. \eqref{sec}, we trace \cite{MR847018} and  define the sequence $ \{u_n(t)\}$ as follows:
		\begin{equation}
		\begin{split}
		\psi(0,u_0(t))&:=\frac{f^\prime(0)}{\kappa(0,0)},\\
		\psi(t,u_{n+1}(t))&:=\frac{f^\prime(t)}{\kappa(t,t)}-\int_{0}^{t} \frac{\kappa_{t}(s,t)}{\kappa(t,t)}\psi(s,u_n(s))\mathrm{d}s, \quad n\geq 0.
		\end{split}
		\end{equation} 
		By the assumptions (\textit{\ref{item3}})  and (\textit{\ref{item4}}), $ \psi(t,u(t)) $ is strictly monotonic continuous function
		with respect to $ u $. So by considering the Inverse Theorem \cite[p. 68]{courant}, $ u_0 $ is well-defined and belongs to  $ H^{m-1}(I)$.
		Now using induction hypothesis,  $ u_n $ is well-defined and belongs to $  H^{m-1}(I). $
		From the assumptions (\textit{\ref{item1}}), (\textit{\ref{item2}}) and (\textit{\ref{item3}}), we deduce that the function
		 \[ \frac{f^\prime(t)}{\kappa(t,t)}-\int_0^t \frac{\kappa_t(t,s)}{\kappa(t,t)}~\psi(s,u_n(s))\mathrm{d}s, \]
	     belongs to $  H^{m-1}(I). $ Hence by the Inverse Theorem,  $ u_{n+1}\in H^{m-1}(I) $. 
		Using the assumptions (\textit{\ref{item4}}) and (\textit{\ref{item5}})
		one can conclude that 
	{	\[\vert u_{n+1}(t)-u_n(t)\vert \leq \Big(\frac{JL}{M}\Big)^n\frac{t^n}{n!}\max\limits_{s \in I} \vert u_1(s)-u_0(s)\vert,\]}
		where $ L $ is the Lipschitz constant in the assumption (\textit{\ref{item5}}) and $J:=\max \big\{ \vert \frac{k_{t}(t,s)}{k(t,t)} \vert ~\big|~ (t,s) \in I \times I \big\}$. Therefore, without loss of generality for $ m>n,$
		\[\vert u_{m}(t)-u_n(t)\vert \leq \sum_{i=n}^{m-1}\vert u_{i+1}(t)-u_i(t)\vert\leq \Vert u_1(t)-u_0(t)\Vert_{\infty}\sum_{i=n}^{m-1}\Big(\frac{JL T}{M}\Big)^{^i}\frac{1}{i!}.\]
		The term $\sum\limits_{i=0}^{\infty}(\frac{JL T}{M})^{^i}\frac{1}{i!}  $ is convergent, so the Cauchy sequence $ \lbrace u_n\rbrace $ is convergent uniformly to
		\[\lim\limits_{n\rightarrow \infty}u_n(t)=u(t),\]
		where $ u(t) $ belongs to $  H^{m-1}(I).$ This result follows from the fact that  $ u_n(t)\in   H^{m-1}(I)$.
\end{proof}
\section{Numerical scheme}\label{prelim}
In this section, we propose an $hp$-version Legendre collocation method for Volterra integral equation of the first kind. To make the paper self-contained, some basic properties of the shifted Legendre polynomial interpolation are introduced in the following subsection.
\subsection{Preliminaries}
{\bf The Legendre-{Gauss} interpolation operator.}   We denote $ \lbrace t_i,w_i\rbrace_{i=0}^{M} $ as the Legendre-{Gauss} quadrature nodes in $ (-1,1)$ and their corresponding weights. Set $ \Lambda:=(-1,1]$ and let $ \mathcal{P}_{M}(\Lambda) $ be the set of polynomials of degree at most $M$. For any function $ \phi \in  \mathcal{P}_{{2M+1}}(\Lambda) $, the following equality could be obtained from the main property of Gauss quadrature,
\begin{equation}\label{eq2.10}
\int_{\Lambda}\phi(t)\mathrm{d}t=\sum_{j=0}^M w_j\phi(t_j).\end{equation}
Thanks to the above equation, for any $ \phi\psi \in \mathcal{P}_{{2M+1}}(\Lambda) $ and  $ \phi \in  \mathcal{P}_{{M}}(\Lambda)$,
\begin{equation}\label{eq11}
(\phi,\psi)=\langle \phi,\psi\rangle_{M},
\end{equation}
where $ (.,.) $ denotes the inner product of $ L^2(\Lambda)$ and the discrete inner product 
$$ \langle u,v\rangle_M:=\sum_{i=0}^{M} w_j u(t_j)v(t_j), \quad \Vert v\Vert_M=\langle v,v\rangle^\frac{1}{2}_M.$$
Let define $ \mathcal{I}_M^{ t}: C(\Lambda)\rightarrow \mathcal{P}_M(\Lambda)$ as the Legendre-Gauss interpolation operator in the $ t $-direction with the following property
\[\mathcal{I}_M^tv(t_j)=v(t_j), \quad 0\leq j\leq M.\]
Regarding the relation \eqref{eq11}, for any $ \phi \in  \mathcal{P}_{{M+1}}(\Lambda),$
\begin{equation}
\label{eqI}(\mathcal{I}_M^tv,\phi)=\langle \mathcal{I}_M^tv, \phi\rangle_M=\langle v, \phi \rangle_M.
\end{equation}
{ Let $ L_i(t) $ be defined as  
\begin{equation*}
	\begin{split}
		L_i(t)=\left\lbrace \begin{array}{lcc}
			l_i(t),&  t\in \Lambda,\\
			0,& \text{o.w},
		\end{array}\right.
	\end{split}
\end{equation*}
where $ l_i(t) $ is the Legendre polynomial of degree $i$.}

 Since the set of Legendre polynomials form an orthogonal complete set in $ L^{2}(\Lambda)$, namely, a function $ v\in L^{2}(\Lambda) $ can be represented as 
\[v(t)=\sum_{i=0}^{\infty} c_iL_i(t),\]
so $  \mathcal{I}_M^t v(t)$ may expand as 
\begin{equation}
 \mathcal{I}_M^t v(t)=\sum_{i=0}^M\hat{v}_iL_i(t),
\end{equation}
by using the orthogonality condition of the Legendre polynomials 
\[\hat{v}_i=\frac{2i+1}{2}( \mathcal{I}_M^t v,L_i)=\frac{2i+1}{2}\langle { v},L_i\rangle_M, \quad i= 0,1, \dots, M.\]
\subsection{Description of the numerical scheme} \label{des of numerical scheme}
For a fixed integer $ N $, let $ I_h:=\lbrace t_n:~0=t_0<t_1< \dots <t_N=T\rbrace$ be  a mesh on $ I$, $ h_n:=t_n-t_{n-1} $ and $h_{\max}=\max\limits_{1\leq n\leq N} h_n  $. Moreover, denote  $ u^{n}(t) $ the solution of Eq. \eqref{asli} on the $ n$-th subinterval of $ I,$ namely,
\begin{equation}\label{eqe}
u^{n}(t)=u(t), \hspace{0.6in}t\in I_n:=(t_{n-1},t_n], \quad n=1,2, \dots, N.
\end{equation}
By above mesh, we rewrite the Eq. \eqref{asli} as
\[\int_{0}^{t_{n-1}} \kappa(s,t) \psi(s,u(s))\mathrm{d}s+\int_{t_{n-1}}^{t} \kappa(s,t) \psi(s,u(s))\mathrm{d}s=f(t),\]
then for any $ t\in I_n $, this equation can be written as
\begin{equation}\label{eq2}
\int_{t_{n-1}}^{t} \kappa(s,t) \psi(s,u^n(s))\mathrm{d}s=f(t)- \sum\limits_{k=1}^{n-1}\int_{t_{k-1}}^{t_{k}} \kappa(\xi ,t) \psi(\xi ,u^{k}(\xi))\mathrm{d}\xi.
\end{equation} 
The problem \eqref{eq2} is converted into an equivalent problem in $ \Lambda:=(-1,1].$ For this aim, we transfer $ t\in I_n$ to $ x\in \Lambda $ by 
\[t=\frac{h_nx+t_{n-1}+t_n}{2},\]
in other words, we have 
\begin{equation}\label{eq3}
\begin{aligned}
\int_{t_{n-1}}^{\frac{h_nx+t_{n-1}+t_n}{2}} \kappa\Big(s,\frac{h_nx+t_{n-1}+t_n}{2}\Big) \psi(s,u^n(s))\mathrm{d}s=f\Big(\frac{h_nx+t_{n-1}+t_n}{2}\Big)\\- \sum\limits_{k=1}^{n-1}\int_{t_{k-1}}^{t_{k}} \kappa\Big(\xi ,\frac{h_nx+t_{n-1}+t_n}{2}\Big) \psi(\xi ,u^{k}(\xi))\mathrm{d}\xi.
\end{aligned}
\end{equation} 
Furthermore, the integral intervals $ I_k $ and $ \Big(t_{n-1},\frac{h_nx+t_{n-1}+t_n}{2}\Big] $ can { be converted} to $ \Lambda $ and $ (-1,x],$ respectively under the following transformation
\[\xi=\frac{{ h_k}\eta+t_{k-1}+t_k}{2},~~~~~s=\frac{h_n\tau+t_{n-1}+t_n}{2}.\]
Hence, Eq. \eqref{eq3} becomes

\begin{equation}\label{eq4}
\begin{split}
&\frac{h_n}{2}\int_{-1}^{x} \kappa\Big(\frac{h_n\tau+t_{n-1}+t_n}{2},\frac{h_nx+t_{n-1}+t_n}{2}\Big)\psi\Big(\frac{h_n\tau+t_{n-1}+t_n}{2},u^n(\frac{h_n\tau+t_{n-1}+t_n}{2})\Big)\mathrm{d}\tau \\
=&f\big(\frac{h_nx+t_{n-1}+t_n}{2}\Big)\\
&- \sum\limits_{k=1}^{n-1} \frac{h_k}{2}\int_{\Lambda} \kappa\Big(\frac{h_k\eta+t_{k-1}+t_k}{2} ,\frac{h_nx+t_{n-1}+t_n}{2}\Big) \psi\Big(
\frac{h_k\eta +t_{k-1}+t_k}{2},u^{k}(\frac{h_k\eta+t_{k-1}+t_k}{2})\Big)\mathrm{d}\eta.
\end{split}
\end{equation} 
Finally, using the linear transform 
{ \begin{equation}\label{tau} \tau=\sigma(x,\theta):=\frac{1+x}{2}\theta+\frac{1-x}{2},\end{equation}}
Eq. \eqref{eq4} reads
\begin{equation}\label{eq5}
\begin{split}
\frac{h_n}{4}(1+x)\int_{\Lambda} \tilde{\kappa}^n(\sigma(x,\theta) ,x) {\tilde{\psi}^n}\big(\sigma(x,\theta) ,\tilde{u}^n(\sigma(x,\theta))\big)\mathrm{d}\theta=\\ { \tilde{f}^n(x)}\vspace*{0.2in}- \sum\limits_{k=1}^{n-1} \frac{h_k}{2}\int_{\Lambda} \tilde{\kappa}^k(\eta ,x) { \tilde{\psi}^k}\big(\eta,\tilde{u}^{k}(\eta)\big)\mathrm{d}\eta,\quad x\in \Lambda,
\end{split}
\end{equation}
where 
\begin{equation}\label{eqt}
\begin{split}
\tilde{u}^k(x)&=u^{k}\big(\frac{h_kx+t_{k-1}+t_k}{2}\big),\\ \quad { \tilde{f}^n(x)}&=f\big(\frac{h_nx+t_{n-1}+t_n}{2}\big),\\
\tilde{\kappa}^{k}(\eta,x)&=\kappa\big(\frac{h_k\eta+t_{k-1}+t_k}{2} ,\frac{h_nx+t_{n-1}+t_n}{2}\big),\\
{\tilde{\psi}^{k}(\eta,x)}&{=\psi\big(\frac{h_k\eta+t_{k-1}+t_k}{2} ,x\big).}
\end{split}
\end{equation}
\subsection{The $ hp $-version of Legendre-Gauss collocation method}
In order to seek a solution $ \tilde{u}^n_{M_n}(x)\in \mathcal{P}_{M_n}(\Lambda)$ of Eq. \eqref{eq5} by $hp$-collocation method, at first step this equation will be fully discretized as
\begin{equation}\label{eq6}
\begin{split}
\mathcal{I}^x_{M_n}\Bigg(\frac{h_n}{4}(1+x)\int_{\Lambda} \mathcal{I}^\theta_{M_n}\Big( \tilde{\kappa}^n(\sigma(x,\theta) ,x) {{ \tilde{\psi}^{n}}}\big(\sigma(x,\theta) ,\tilde{u}^n_{M_n}(\sigma(x,\theta))\big)\Big) \mathrm{d}\theta \Bigg)\vspace*{0.2in}\\=\mathcal{I}^x_{M_n}\left( \tilde{f}^{n}(x)\vspace*{0.2in}- \sum\limits_{k=1}^{n-1} \frac{h_k}{2}\int_{\Lambda}\mathcal{I}^\eta_{M_k}\big( \tilde{\kappa}^k(\eta ,x) {{ \tilde{\psi}^{k}}}(\eta,\tilde{u}^{k}_{M_k}(\eta))\big)\mathrm{d}\eta \right),\quad x\in \Lambda  ,
\end{split}
\end{equation}
where
\begin{equation}\label{eql}
\begin{split}
&{\tilde{u}^n_{M_n}(x)}=\sum_{p=0}^{M_n}\hat{u}^{n}_p L_{p}(x), \\
&\mathcal{I}^x_{M_n}\mathcal{I}^\theta_{M_n}\Big((1+x)\tilde{\kappa}^n(\sigma(x,\theta) ,x) {{ \tilde{\psi}^{n}}}\big(\sigma(x,\theta) ,\tilde{u}^n_{M_n}(\sigma(x,\theta))\big)\Big)=\sum_{p,q=0}^{M_n}a^{n}_{pq} L_{p}(x)L_q(\theta),\\
&\mathcal{I}^x_{M_n}{\mathcal{I}^\eta_{M_k}}\Big(\tilde{\kappa}^k(\eta ,x) {{ \tilde{\psi}^{k}}}\big(\eta,\tilde{u}^{k}_{M_k}(\eta)\big)\Big)=\sum_{p=0}^{M_n}\sum_{q=0}^{M_k}b^k_{pq} L_{p}(x)L_q(\eta),
\end{split}
\end{equation}
and
\begin{equation}\label{eq8}
{\mathcal{I}^x_{M_n}\tilde{f}^{n}}(x)=\sum_{p=0}^{M_n}c^{n}_p L_{p}(x).
\end{equation}
Then, we have
\begin{equation}\label{eq9}
\begin{aligned}
\int_\Lambda\mathcal{I}^x_{M_n}\mathcal{I}^\theta_{M_n}\big((1+x)\tilde{\kappa}^n(\sigma(x,\theta) ,x) {{ \tilde{\psi}^{n}}}(\sigma(x,\theta) ,\tilde{u}^n_{M_n}(\sigma(x,\theta)))\big)\mathrm{d}\theta &=\sum_{p,q=0}^{M_n}a^{n}_{pq} L_{p}(x)\int_\Lambda L_q(\theta)\mathrm{d}\theta\\&=2\sum_{p=0}^{M_n}a^{n}_{p0} L_{p}(x),
\end{aligned}
\end{equation}
and similarly,
\begin{equation}\label{eq10}
\begin{aligned}
\int_\Lambda \mathcal{I}^x_{M_n}\mathcal{I}^\eta_{{M_k}}\Big(\tilde{\kappa}^k(\eta ,x) { \tilde{\psi}^{k}}(\eta,\tilde{u}^{k}_{M_k}(\eta))\Big)\mathrm{d}\eta &=\sum_{p=0}^{M_n}\sum_{q=0}^{M_k}b^k_{pq} L_{p}(x)\int_\Lambda L_q(\eta)\mathrm{d}\eta\\&=2\sum_{p=0}^{M_n}b^{k}_{p0} L_{p}(x).
\end{aligned}
\end{equation}
We denote the Legendre-{ Gauss} quadrature nodes  and the corresponding weights in $ (-1,1) $ by $\big \{x_{k,i},w_{k,i}\big\}_{i=0}^{M_k} $ which are related to $ k$-th subinterval. It can be determined from Eqs. \eqref{eq8}-\eqref{eq10} that
\begin{equation}
\begin{aligned}
\hat{u}_p^n=&\frac{2p+1}{2}\sum_{i=0}^{M_n}\tilde{u}^n_{M_n}(x_{n,i})L_p(x_{n,i})w_{n,i},\\
a_{p0}^n=&\frac{2p+1}{4}\sum_{i,j=0}^{M_n}(1+x_{n,i})\tilde{\kappa}^n\big(\sigma(x_{n,i},x_{n,j}) ,x_{n,i}\big) {{ \tilde{\psi}^{n}}}\big(\sigma(x_{n,i},x_{n,j}) ,\tilde{u}^n_{M_n}(\sigma(x_{n,i},x_{n,j}))\big)\\&L_p(x_{n,i})w_{n,i}w_{n,j},\\
b_{p0}^k=&\frac{2p+1}{4}\sum_{i=0}^{M_n}\sum_{j=0}^{M_k}\tilde{\kappa}^k({x_{k,j},x_{n,i}}) {\tilde{\psi}^{k}\big(x_{k,j} ,\tilde{u}^k_{M_k}(x_{k,j})}\big)L_p(x_{n,i})w_{n,i}w_{{k},j},\\
c_p^n=&\frac{2p+1}{2}\sum_{i=0}^{M_n}\tilde{f}^n(x_{n,i})L_p(x_{n,i})w_{n,i}.
\end{aligned}
\end{equation}
{ With Eqs. \eqref{eq8}-\eqref{eq10}, Eq. \eqref{eq6} reads}
\[0=\sum_{p=0}^{M_n}c^{n}_p L_{p}(x){-}\sum_{p=0}^{M_n}\tilde{a}^{n}_p L_{p}(x){-}\sum_{p=0}^{M_n}\tilde{b}^{n}_p L_{p}(x),\]
where
\[\tilde{a}^{n}_p=\frac{1}{2}h_na_{p0}, \qquad \tilde{b}^{n}_p=\sum_{k=1}^{n-1}h_k b^{k}_{p0}.\]
Consequently, we compare the expansion coefficient to obtain 
\begin{equation}\label{eq7}
0=\tilde{a}_p^n+\tilde{b}_p^{n}-c^n_p,\quad 0\leq p \leq M_n.
\end{equation}
 To evaluate  the unknown coefficients $ \hat{u}^n_p $ for any given $ n $, we solve the nonlinear system \eqref{eq7} with the Newton  iteration method. { Finally the approximate solution by Eqs. \eqref{eqe}, \eqref{eqt} and \eqref{eql} is 
\begin{equation}\label{unm}
u_{_M}^{N}(t)=\sum_{k=1}^{N}u^k(t)=\sum_{k=1}^{N}\tilde{u}_{M_k}^k\Big(\frac{2t-t_{k-1}-t_k}{h_k}\Big)=\sum_{k=1}^{N}\sum_{p=0}^{M_k}\hat{u}^k_p L_p\Big(\frac{2t-t_{k-1}-t_k}{h_k}\Big), \qquad t \in I.
\end{equation}}
{\begin{remark}\label{rem2}
For the linear case of Eq. \eqref{asli}, all mentioned relations are valid with $ \psi(t,u(t))=u(t).$ Therefore, the unknown coefficients  $ \hat{u}^n_p $ for any given $ n $ can be obtained by the following linear system of equations
\begin{equation}\label{eqr}
A{\bf u}={\bf b}+{\bf c},
\end{equation}
where the entries of the matrix $ A=[a_{i,j}]_{i,j=0}^{M_n}$ are defined by
\[
a_{i,j}=\frac{2p+1}{8}h_n\sum_{i,j=0}^{M_n}(1+x_{n,i})\tilde{\kappa}^n\big(\sigma(x_{n,i},x_{n,j}) ,x_{n,i}\big) L_q\big(\sigma(x_{n,i},x_{n,j})\big)L_p(x_{n,i})w_{n,i}w_{n,j},
\]
\end{remark}
and 
\[{\bf u}=(\hat{u}^n_0, \dots ,\hat{u}^n_{M_n})^T,\quad {\bf b}=(\tilde{b}^n_0, \dots ,\tilde{b}^n_{M_n})^T,\quad {\bf c}=(c^n_0, \dots,c^n_{M_n})^T. \]
\begin{remark}
		In a general statement, $hp$-collocation method can be categorized as a projection method with an appropriate uniformly bounded projector.  In this problem, due to the smoothness of the kernel and the right hand side function $f(t)$, the existence of the solution for Eq. \eqref{eq6} can be inferred from the general  framework for analysis of the projection methods,  for more details see \cite{ MR0502103, krasno, kress}.
\end{remark}}
\section{Error analysis}\label{error analysis}
In this section, we should give functional framework and for this aim some weighted Sobolev spaces are defined. 

Let us define the weight function $ \chi^{(\alpha, \beta)}(x)=(1-x)^{\alpha}(1+x)^{\beta} $ for $ \alpha, \beta >-1 $. For $ r \in \mathbb{N}$, $ H^{r}_{\chi^{(\alpha,\beta)}}(\Lambda) $ is a weighted Sobolev space defined by
\[
H^{r}_{\chi^{(\alpha,\beta)}}(\Lambda) = \Big\{ v ~|~ v ~\text{is measurable and } \Vert v \Vert_{r, \chi^{(\alpha, \beta)}} < \infty \Big\},
\]
equipped with the following norm  
\[
\Vert v \Vert_{r, \chi^{(\alpha, \beta)}} = \Big( \sum_{k=0}^{r} \Vert \partial^{{ k}}_{x}v \Vert^{2}_{\chi^{(\alpha+k,\beta+k)}} \Big)^\frac{1}{2},
\]
and semi-norm
\[
\vert v \vert_{r, \chi^{(\alpha, \beta)}} = \Vert \partial^{r}_{x}v \Vert_{\chi^{(\alpha +r, \beta+r)}},
\]
where  $ \Vert . \Vert_{\chi^{(\alpha, \beta)}}  $ is an appropriate norm for the space $ L^{2}_{\chi^{(\alpha, \beta)}}(\Lambda)$.
{ Throughout this paper, we denote  $ \Vert.\Vert $ as $ L^2 $-norm and $ M_{\min}=\min_{1\leq n\leq N}M_n. $}
\begin{lemma}(\cite{hack})\label{lem2} (Gr\"{o}nwall inequality)
	Assume that there are numbers $ \alpha,~ \beta_l\geq 0~ (l=0,1, \dots, n-1)$ and $ 0 \leq M_{_0} <1 $ such that 
	\[ 0 \leq \varepsilon_n \leq \alpha+\sum_{l=0}^{n-1} \beta_l\varepsilon_l+M_{_0} \varepsilon_n, \quad n\geq 1.\]
	Then the quantities $ \varepsilon_n $ fulfill the following estimate for $ n\geq 0$
	\[\varepsilon_n\leq \frac{\alpha}{1-M_{_0}}\exp(\sum_{l=0}^{n-1}\frac{\beta_l}{1-M_{_0}}).\]
\end{lemma}
In the following, some theoretical results regarding the convergence of the method are stated.
\begin{lemma}\label{lem1}
	For any $\tilde{v}\in H^{^m}_{\chi_{(0,0)}}(\Lambda)$ with integer $1\leq m \leq M_n+1$ and $ 1\leq n \leq N, $
	\[\Vert \tilde{v}-\mathcal{I}_{M_n}^x \tilde{v}\Vert^2\leq c M_{n}^{-2m}\Vert \partial^m_x \tilde{v}\Vert_{\chi^{(m,m)}}^2\leq ch_n^{2m-1}M_n^{-2m}\Vert \partial^m_t v\Vert_{L^2(I_n)}^2,\]
	where $ \tilde{v}(x)=v(t)\Big \vert_{t=\frac{h_nx+t_{n-1}+t_n}{2}}. $
\end{lemma}
\begin{proof}
	First inequality is proved in \cite{guo}. For the second {inequality}, we have
	{\begin{equation}
	\begin{split}
	\Vert \tilde{v}-\mathcal{I}_{M_n}^x \tilde{v}\Vert^2 &\leq c M_n^{-2m}\int_{\Lambda} (\partial^m_x \tilde{v})^2 (1-x^2)^m \mathrm{d}x \\
	&=c M_n^{-2m}(\frac{h_n}{2})^{-2m-1}\int_{I_n} (\partial^m_t \tilde{v})^2 (\partial^m_x t)^2 (t-t_{n-1})^m(t_n-t)^m \mathrm{d}t\\
	&= ch_n^{-1}M_n^{-2m}\int_{I_n} (\partial^m_t v)^2 (t-t_{n-1})^m(t_n-t)^m \mathrm{d}t\\& \leq ch_n^{-1}M_n^{-2m}\Vert \partial^m_t v\Vert_{L^2(I_n)}^2\max_{t\in I_n}\big((t-t_{n-1})^m(t_n-t)^m\big)\\&\leq ch_n^{2m-1}M_n^{-2m}\Vert \partial^m_t v\Vert_{L^2(I_n)}^2.\\
	\end{split}
	\end{equation}
	Note that $ t=\frac{h_nx+t_{n-1}+t_n}{2}$ and $\partial^m_x t=(\frac{h_n}{2})^m$. }
\end{proof}
\begin{theorem}\label{theorem11}
	Let $\tilde{u}^n$ be the solution of Eq. \eqref{eq5} { under the hypothesis of Theorem 1 }and $\tilde{u}^n_{M_n}$ be the solution of Eq. {\eqref{eq6}}. Therefore, assume that $\kappa(s,t)\in C^{^m}(I\times I)$,
	$f(t)\Big\rvert_{I_n}\in H^{^{m}}(I_n)$, 
	$u(t)\Big\rvert_{I_n}\in H^{^m}(I_n)$
	and 
	$\psi:I_n\times H^{^m}(I_n) \rightarrow H^{^m}(I_n)$ for $n=1,2, \dots ,N, {~m\leq M_{\min}+1}$
	and $\psi(.,u) $ fulfills the Lipschitz condition with respect to the second variable, i.e.,
	\begin{equation}\label{lip}
	\vert \psi(.,u_1)-\psi(.,u_2)\vert\leq \gamma \vert u_1-u_2 \vert,\quad \gamma \geq 0.
	\end{equation}
	Then, for any $1\leq n\leq N$,
	\[B_1(x)=B_2(x)+B_3(x),\] 
	with 
	\begin{equation}
	\begin{aligned}
	\Vert B_1(x)\Vert^2&\leq c h_n^{2m-1}M_n^{-2m}\Vert\partial^{m}_t f\Vert_{L^2(I_n)}^2+c T \sum_{k=1}^{n-1}\Big(\gamma^2 h_k \Vert e_k\Vert^2+c h_k^{2m}M_k^{-2m}\big(\gamma^2\Vert \partial^{m}_t u\Vert^2_{L^2(I_k)}\\&+\Vert \psi(.,u)\Vert^2_{H^m(I_k)}\big)\Big)+c T h_n^{2m}M_n^{-2m}\sum_{k=1}^{n-1}\Vert \psi(.,u)\Vert^2_{H^1(I_k)},
	\end{aligned}
	\end{equation}
	where
	\begin{equation}\label{eq14}
	\begin{split}
	B_1(x)&=\frac{h_n}{4}\Big((1+x) \tilde{\kappa}^n(\sigma(x,.) ,x) ,{{\tilde{\psi}^{n}}}\big(\sigma(x,.) ,\tilde{u}^n(\sigma(x,.))\big)\Big)-\frac{h_n}{4}\mathcal{I}_{M_n}^x\Big\langle(1+x)\tilde{\kappa}^n(\sigma(x,.) ,x), {{ \tilde{\psi}^{n}}}\big(\sigma(x,.) ,\tilde{u}^n_{M_n}(\sigma(x,.))\big)\Big\rangle_{{M_n}},\vspace{0.1in}\\
	B_2(x)&=\tilde{f}^n(x)-\mathcal{I}_{M_n}^x\tilde{f}^n(x),\vspace{0.1in}\\
	B_3(x)&=-\sum\limits_{k=1}^{n-1} \frac{h_k}{2}\Big( \tilde{\kappa}^k(. ,x), {{ \tilde{\psi}^{k}}}\big(.,\tilde{u}^{k}(.)\big)\Big)+\sum\limits_{k=1}^{n-1} \frac{h_k}{2}\mathcal{I}_{M_n}^x \Big\langle \tilde{\kappa}^k(.,x), {{ \tilde{\psi}^{k}}}\big(.,\tilde{u}_{M_k}^{k}(.)\big)\Big\rangle_{M_k},
	\end{split}
	\end{equation}
	and $ e_k=\tilde{u}^{k}-\tilde{u}^{k}_{M_k},\quad 1\leq k\leq N$.
\end{theorem}
\begin{proof}
	The main problem \eqref{asli} can be converted into the interval $ \Lambda $ as Eq. \eqref{eq5}. In the present scheme, we approximate it by Eq. {\eqref{eq6}}. Regarding Eq. \eqref{eq11}, we have from Eqs. \eqref{eq5} and {\eqref{eq6}} that
	\begin{equation}\label{eq12}
	\begin{split}
	\frac{h_n}{4}\Big((1+x) \tilde{\kappa}^n(\sigma(x,.) ,x) ,{{ \tilde{\psi}^{n}}}\big(\sigma(x,.) ,\tilde{u}^n(\sigma(x,.))\big)\Big)=\tilde{f}^n(x)- \sum\limits_{k=1}^{n-1} \frac{h_k}{2}\Big( \tilde{\kappa}^k(. ,x), {{ \tilde{\psi}^{k}}}\big(.,\tilde{u}^{k}(.))\big)\Big),\quad x\in \Lambda,
	\end{split}
	\end{equation}
	and
	\begin{equation}\label{eq13}
	\begin{split}
	\frac{h_n}{4}\mathcal{I}_{M_n}^x\Big\langle(1+x)\tilde{\kappa}^n(\sigma(x,.) ,x) ,{{ \tilde{\psi}^{n}}}\big(\sigma(x,.) ,\tilde{u}_{M_n}^n(\sigma(x,.))\big)\Big\rangle_{{M_n}}=\mathcal{I}_{M_n}^x\tilde{f}^n(x)- \sum\limits_{k=1}^{n-1} \frac{h_k}{2}\mathcal{I}_{M_n}^x\Big\langle \tilde{\kappa}^k(.,x), {{ \tilde{\psi}^{k}}}\big(.,\tilde{u}_{M_k}^{k}(.)\big)\Big\rangle_{{M_k}}.
	\end{split}
	\end{equation}
	By subtracting  \eqref{eq13} from \eqref{eq12}, we have
	\begin{equation}\label{b1}{	B_1(x)=B_2(x)+B_3(x),}
	\end{equation}
	where $B_1(x)$, $B_2(x)$ and $B_3(x)$ are defined in \eqref{eq14}.
	
	In order to obtain an estimate error for the term $B_1$, we need error bounds for $ \Vert B_i\Vert,~ i=2,3.$ First using Lemma \ref{lem1}, we infer that
	\begin{equation}
	\Vert B_2(x)\Vert^2= \Vert \tilde{f}^n(x)-\mathcal{I}_{M_n}^x\tilde{f}^n(x)\Vert^2 
	\leq c h_n^{2m-1}M_n^{-2m}\Vert\partial^{m}_t f(t)\Vert_{L^2(I_n)}^2.
	\end{equation}
With the same argument in \cite{sheng} about $ \Vert B_3 \Vert,$ one can conclude that 	{(see Appendix B)}
	\begin{equation}\label{eqB}
	\begin{aligned}
	\Vert B_3(x)\Vert^2&\leq c T \sum_{k=1}^{n-1}\Big(\gamma^2 h_k \Vert e_k\Vert^2+c h_k^{2m}M_k^{-2m}\Big(\gamma^2\Vert \partial^{m}_t u\Vert^2_{L^2(I_k)}\\&+\Vert \psi(.,u)\Vert^2_{H^m(I_k)}\Big)\Big)+c T h_n^{2m}M_n^{-2m}\sum_{k=1}^{n-1}\Vert \psi(.,u)\Vert^2_{H^1(I_k)}.
	\end{aligned}
	\end{equation}
	So the desire{d} result follows from $ \Vert B_1(x)\Vert^2\leq 2\Big(\Vert B_2(x)\Vert^2+ \Vert B_3(x)\Vert^2\Big).$
\end{proof}
\begin{theorem}\label{te1}
	Assume that the Fr{\'e}chet derivative of the operator $ \mathcal{K}u $ with respect to $ u $ satisfies at
$\vert (\mathcal{K}^\prime u)(t) \vert \geq \alpha > 0,$ 
	then under the hypothesis of the Theorem \ref{theorem11}, for sufficiently small $ h_{\max} $ the following error estimate is obtained
	\begin{equation}
	\begin{aligned}
	\Vert e_n \Vert^2=\Vert  \tilde{u}^n -\tilde{u}^n_{M_n}\Vert^2 &\leq \frac{c}{\delta^2}\exp(c\gamma^2 T^2)\Bigg(h_n^{2m-1}M_n^{-2m}\Vert\partial^{m}_t f\Vert_{L^2(I_n)}^2+ 	{h_n^{2m-1}}M_n^{-2m}\Big(\gamma^2h_n^2\Vert \partial^{m}_t u\Vert^2_{L^2(I_n)}\\
	&~~~+\Vert \psi(.,u)\Vert^2_{H^m(I_n)}\Big)+ h_n^{2m-1}M^{-2m}_n\Vert \psi(.,u_{_M}^{N})\Vert^2_{H^m(I_n)}\\
	& ~~~+T \sum_{k=1}^{n-1}\Big(c h_k^{2m}M_k^{-2m}\Big(\gamma^2\Vert \partial^{m}_t u\Vert^2_{L^2(I_k)}+\Vert \psi(.,u)\Vert^2_{H^m(I_k)}\Big)\\
	&~~~+ h_n^{2m}M_n^{-2m}\sum_{k=1}^{n-1}\Vert \psi(.,u)\Vert^2_{H^1(I_k)}\Big)\Bigg).
	\end{aligned}
	\end{equation}
\end{theorem}
\begin{proof}
	For convenience, let
	\begin{equation}
	F(x,\tau,\tilde{u}(\tau)):= \frac{h_{n}}{2}\tilde{\kappa}^{n}{(\tau,x)} {{ \tilde{\psi}^{n}}}(\tau,\tilde{u}(\tau)),\hspace{0.2in}\tau \in (-1,x].
	\end{equation} and $\mathcal{G}{\tilde{u}}(x):=\int_{-1}^{x}F(x,\tau,{\tilde{u}(\tau)}) \mathrm{d}\tau.$
	Under the mean value theorem \cite[p. 229]{atkinson}, we have
	\begin{equation}
	\int_{-1}^{x}F(x,\tau,\tilde{u}^{n}(\tau)) \mathrm{d}\tau  - \int_{-1}^{x}F(x,\tau,\tilde{u}_{M_n}^{n}(\tau))\mathrm{d}\tau= \mathcal{G}^\prime ({\xi_x})\big(\tilde{u}^{n}(x) - \tilde{u}_{M_n}^{n}(x)\big),
	\end{equation}
	where $ {\xi_x}\in (\min\lbrace \tilde{u}^{n},\tilde{u}_{M_n}^{n}\rbrace,\max\lbrace \tilde{u}^{n},\tilde{u}_{M_n}^{n}\rbrace ) $ and $ \mathcal{G}^\prime $ denotes the Fr{\'e}chet derivative, namely, 
	\[ \mathcal{G}^\prime (\tilde{u}) h(x)=\int_{-1}^x \tilde{\kappa}^{n}(\tau,x) \frac{\partial {{{ \tilde{\psi}^{n}}}}(\tau,\tilde{u}(\tau)) }{\partial \tilde{u}} h(\tau) \mathrm{d}\tau. \]
	{ It is seen that 
		\[
		\mathcal{K}'(u)h(t) = \int_{0}^{t}k(t,s)\frac{\partial}{\partial u}\psi(s,u(s))h(s)\mathrm{d}s,
		\]
		so the assumption of $\vert (\mathcal{K}^\prime u) (t)\vert \gg 0$ results that  $\delta:=\vert \mathcal{G}^\prime (\tilde{u}) \vert \gg 0.$
	Therefore, in order to obtain an upper bound for the error, one can write}
	\begin{equation}\label{eq15}
	\Big \vert\tilde{u}^{n}(x) - \tilde{u}_{M_n}^{n}(x) \Big \vert \leq \frac{1}{\delta} \Big \vert\int_{-1}^{x}F(x,\tau,\tilde{u}^{n}(\tau) \mathrm{d}\tau  - \int_{-1}^{x}F(x,\tau,\tilde{u}_{M_n}^{n}(\tau))\mathrm{d}\tau\Big|.
	\end{equation}
	Since
	\begin{equation}
	\int_{-1}^{x}F(x,\tau  ,\tilde{u}^{n}(\tau)) \mathrm{d}\tau = \frac{h_{n}}{4}\int_\Lambda (1 + x)\tilde{\kappa}^{n}({\sigma (x,\eta),x}){{ \tilde{\psi}^{n}}}\Big(\sigma(x,\eta),\tilde{u}^{n}( \sigma(x,\eta))\Big)\mathrm{d}\eta,
	\end{equation}
	then from \eqref{eq15}, we deduce that
	\begin{equation}
	\begin{aligned}
	\Big \vert e_n(x)\Big \vert=\Big \vert \tilde{u}^{n}(x) - \tilde{u}_{M_n}^{n}(x) \Big \vert &\leq \frac{1}{\delta} \Big \vert \frac{h_{n}}{4}(1 + x)\Big(\int_{ \Lambda }\tilde{\kappa}^{n}({{\sigma (x,\eta),x) \tilde{\psi}^{n}}}\big(\sigma(x,\eta),\tilde{u}^{n}( \sigma(x,\eta))\big)\mathrm{d}\eta  \\&\quad-\int_{ \Lambda }\tilde{\kappa}^{n}({{\sigma (x,\eta),x) \tilde{\psi}^{n}}}\big(\sigma(x,\eta),\tilde{u}_{M_n}^{n}( \sigma(x,\eta))\big)\mathrm{d}\eta \Big)\Big \vert\\&	
	\leq \frac{1}{\delta}\Big(\vert B_{1}(x)\vert + E_{1}(x) + E_{2}(x)\Big),
	\end{aligned}
	\end{equation}
	where $ B_1(x) $ is defined in \eqref{eq14} and
	\begin{equation}\label{E12}
	\begin{aligned}
	E_1(x)&=\frac{h_{n}}{4}\Big \vert  \mathcal{I}_{M_n}^{x}\Bigg((1 + x)\int_{\Lambda} \mathcal{I}_{M_n}^{\eta}\Big(\tilde{\kappa}^n ({{\sigma (x,\eta),x) \tilde{\psi}^{n}}}\big(\sigma(x,\eta),\tilde{u}^{n}_{M_n}( \sigma(x,\eta))\big)\Big)\mathrm{d}\eta{\Bigg)}\\&- (1 + x)\int_{\Lambda} \mathcal{I}_{M_n}^{\eta} \Big(\tilde{\kappa}^n ({{\sigma (x,\eta),x) \tilde{\psi}^{n}}}\big(\sigma(x,\eta),\tilde{u}^{n}_{M_n}( \sigma(x,\eta))\big)\Big)\mathrm{d}\eta\Big \vert, \\
	E_2(x)&=\frac{h_{n}}{4}\Big \vert (1 + x)\Big(\int_{\Lambda} \mathcal{I}_{M_n}^{\eta} \Big(\tilde{\kappa}^{n}({{\sigma (x,\eta),x) \tilde{\psi}^{n}}}\big(\sigma(x,\eta),\tilde{u}^{n}_{M_n}( \sigma(x,\eta))\big)\Big)\mathrm{d}\eta\\&-
	\int_\Lambda \tilde{\kappa}^{n}({{\sigma (x,\eta),x) \tilde{\psi}^{n}}}\big(\sigma(x,\eta),\tilde{u}^{n}_{M_n}( \sigma(x,\eta))\big)\mathrm{d}\eta\Big) \Big \vert.
	\end{aligned}
	\end{equation}
	Thus 
	\begin{equation}
	\Vert e_n(x)\Vert^2\leq \frac{3}{\delta^2}\Big( \Vert B_1(x) \Vert^2+ \Vert E_1(x) \Vert^2+ \Vert E_2(x)\Vert^2\Big).
	\end{equation}
	In order to estimate the term $ \Vert e_n\Vert^2 $, we need the error bound for $ \Vert E_i\Vert,$ for $i=1,2. $ 
	Owing to Lemma \ref{lem1} and the same discussion in \cite{sheng} 	{(see Appendix A)}, we have
	\begin{equation}\label{eqE1}
	\begin{aligned}
	\Vert E_1 \Vert^2&=	{ \frac{h_{n}^2}{4}\Bigg\Vert(\mathcal{I}_{M_n}^{x}-\mathcal{I})\Bigg(\frac{1 + x}{2}\int_{\Lambda} \mathcal{I}_{M_n}^{\eta}\Big(\tilde{\kappa}^n (\sigma (x,\eta),x){{ \tilde{\psi}^{n}}}\big(\sigma(x,\eta),\tilde{u}^{n}_{M_n}( \sigma(x,\eta))\big)\Big)\mathrm{d}\eta\Bigg)\Big \Vert^2}\\&\leq \frac{d\gamma^2h_n^2}{3}\Big \Vert e_n\Big \Vert^2+c h_n^{2m-1}M_n^{-2m}\Big(\gamma^2h_n^2\Big \Vert \partial^{m}_t u\Big \Vert^2_{L^2(I_n)}+\Big \Vert \psi(.,u)\Big \Vert^2_{H^m(I_n)}\Big),
	\end{aligned}
	\end{equation}
	and
	\begin{equation}\label{eqE2}
	\begin{aligned}
	\Vert E_2\Vert^2&={\frac{h_{n}^2}{4}\Big \Vert\frac{1 + x}{2} \int_{\Lambda} (\mathcal{I}_{M_n}^{\eta}-\mathcal{I}) \Big(\tilde{\kappa}^{n}(\sigma (x,\eta),x){{ \tilde{\psi}^{n}}}\big(\sigma(x,\eta),\tilde{u}_{M_n}^{n}( \sigma(x,\eta))\big)\Big)\mathrm{d}\eta\Big\Vert^2}\\&\leq c h_n^{2m+1}M^{-2m}_n\Big \Vert \psi(.,u_{_M}^{N})\Big \Vert^2_{H^m(I_n)},
	\end{aligned}
	\end{equation}
	where the constant $ d $ depends on the term $ \max\limits_{(s,t)\in \Omega}\vert \kappa(s,t)\vert $ and $ \gamma $ is the Lipschitz constant. Consequently,
	\begin{equation}
	\begin{aligned}
	(1-\frac{d\gamma^2h_n^2}{\delta^2}) \Vert e_n \Vert^2 &\leq \frac{1}{\delta^2}\Bigg( c h_n^{2m-1}M_n^{-2m}\Vert\partial^{m}_t f\Vert_{L^2(I_n)}^2+c T \sum_{k=1}^{n-1}\Big(\gamma^2 h_k \Vert e_k\Vert^2+c h_k^{2m}M_k^{-2m}\Big(\gamma^2\Vert \partial^{m}_t u\Vert^2_{L^2(I_k)}\\&+\Vert \psi(.,u)\Vert^2_{H^m(I_k)}\Big)\Big)+c T h_n^{2m}M_n^{-2m}\sum_{k=1}^{n-1}\Vert \psi(.,u)\Vert^2_{H^1(I_k)}+c h_n^{2m-1}M_n^{-2m}\Big(\gamma^2h_n^2\Vert \partial^{m}_t u\Vert^2_{L^2(I_n)}\\&+\Vert \psi(.,u)\Vert^2_{H^m(I_n)}\Big)+c h_n^{2m+1}M^{-2m}_n\Vert \psi(.,u_{_M}^{N})\Vert^2_{H^m(I_n)}\Bigg).
	\end{aligned}
	\end{equation}
	We assume that $ h_{\max} $ is sufficiently small such that
	\[ \frac{d\gamma^2h_{ \max}^2}{\delta^2}\leq \beta <1,\]
now using Lemma \ref{lem2}, we have
	\begin{equation}
	\begin{split}
	\Vert e_n \Vert^2 &\leq \frac{c}{\delta^2}\exp(c\gamma^2 T^2)\Bigg(h_n^{2m-1}M_n^{-2m}\Vert\partial^{m}_t f\Vert_{L^2(I_n)}^2+ h_n^{2m-1}M_n^{-2m}\Big(\gamma^2h_n^2\Vert \partial^{m}_t u\Vert^2_{L^2(I_n)}\\&+\Vert \psi(.,u)\Vert^2_{H^m(I_n)}\Big)+ h_n^{2m-1}M^{-2m}_n\Vert \psi(.,u_{_M}^{N})\Vert^2_{H^m(I_n)}+T \sum_{k=1}^{n-1}\Big (h_k^{2m}M_k^{-2m}\Big(\gamma^2\Vert \partial^{m}_t u\Vert^2_{L^2(I_k)}\\&+\Vert \psi(.,u)\Vert^2_{H^m(I_k)}\Big)\Big)+ h_n^{2m}M_n^{-2m}\sum_{k=1}^{n-1}\Vert \psi(.,u)\Vert^2_{H^1(I_k)}\Bigg),
	\end{split}
	\end{equation}
	hence, the desired result is obtained.
\end{proof}
\begin{theorem}\label{thm8}
Assume that $ u(t) $ be the exact solution of Eq. \eqref{asli} and $ u_{_M}^N(t) $ be the global approximate solution obtained from Eq. \eqref{unm}. Under the hypothesis of Theorem \ref{te1}, the following error estimate can be derived as
\begin{equation}
\begin{aligned}
\Vert u-u_{_M}^N\Vert_{L^2(I)}&\leq \frac{c}{\delta^2}\exp(c\gamma^2 T^2)h_{\max}^{m}M_{\min}^{-m}\Big( \Vert\partial^{m}_t f\Vert_{L^2(I)}+\gamma (1+T) \Vert \partial^{m}_t u\Vert_{L^2(I)}\\&+T\Vert \psi(.,u)\Vert_{H^m(I)}+\Vert \psi(.,u^N_{_M})\Vert_{H^m(I)} \Big).
\end{aligned}
\end{equation}
\end{theorem}
\begin{proof}
The global convergence error of the approximate solution $ u_{_M}^N(t) $ which is given by 
\[ u_{_M}^N(t)\rvert_{t\in I_n}=\tilde{u}^n_{M_n}(x)\Big\rvert_{{x=}\frac{2t-t_{n-1}-t_n}{h_n}}, \quad 1\leq n\leq N, \]
and the exact solution  $ u(t) $ which is fulfilled in
\[ u(t)\rvert_{t\in I_n}=\tilde{u}^n(x)\Big\rvert_{{x=}\frac{2t-t_{n-1}-t_n}{h_n}}, \quad 1\leq n\leq N,\]
can be easily obtained using Theorem \ref{te1}  and the following formula
\[\Vert u-u_{_M}^N\Vert^2_{L^2(I)}=\frac{1}{2}\sum_{n=1}^{N}h_n\Vert e_n\Vert^2.\]
Therefore,
	\begin{equation}\label{uuu}
	\begin{aligned}
	\Vert   u-u_{_M}^N\Vert^2 &\leq \frac{c}{\delta^2}\exp(c\gamma^2 T^2)\sum_{n=1}^{N}\Big(h_n^{2m}M_n^{-2m}\Vert\partial^{m}_t f\Vert_{L^2(I_n)}^2+ h_n^{2m}M_n^{-2m}\Big(\gamma^2h_n^2\Vert \partial^{m}_t u\Vert^2_{L^2(I_n)}\\&+\Vert \psi(.,u)\Vert^2_{H^m(I_n)}\Big)+ h_n^{2m}M^{-2m}_n\Vert \psi(.,u_{_M}^{N})\Vert^2_{H^m(I_n)}+Th_n \sum_{k=1}^{n-1}\Big (h_k^{2m}M_k^{-2m}\Big(\gamma^2\Vert \partial^{m}_t u\Vert^2_{L^2(I_k)}\\&+\Vert \psi(.,u)\Vert^2_{H^m(I_k)}\Big)\Big)+ h_n^{2m+1}M_n^{-2m}\sum_{k=1}^{n-1}\Vert \psi(.,u)\Vert^2_{H^1(I_k)}\Big).
	\end{aligned}
	\end{equation}
All terms of the above error bound could be simplified  using $ h_{\max} $ and $ M_{\min} $ as follows
\[\sum_{n=1}^{N}h_{n}^{2m}M_{n}^{-2m}\Vert\partial^{m}_t f\Vert_{L^2(I_n)}^2\leq h_{\max}^{2m} M_{min}^{-2m}\Vert\partial^{m}_t f\Vert_{L^2(I)}^2.\]
similarly,
 \[\sum_{n=1}^{N} \gamma^2h_{n}^{2m+2}M^{-2m}_n {\Vert  \partial^{m}_t  u\Vert ^2_{L^2(I_n)}} \leq \gamma^2 h_{\max}^{2m+2} M_{\min}^{-2m}\Vert \partial^{m}_t u \Vert_{L^2(I)}^2,\]
Also the following inequalities can be proved 
\[\sum_{n=1}^{N} h_n^{2m}M^{-2m}_n\Vert \psi(.,u)\Vert^2_{H^m(I_n)}\leq h_{\max}^{2m} M_{\min}^{-2m}\Vert\psi(.,u)\Vert^2_{H^m(I)},\]
and
\[\sum_{n=1}^{N} h_n^{2m}M^{-2m}_n\Vert \psi(.,u_{_M}^{N})\Vert^2_{H^m(I_n)}\leq h_{\max}^{2m} M_{\min}^{-2m}\Vert\psi(.,u_{_M}^{N})\Vert^2_{H^m(I)}.\]
Furthermore, we  obtain
\[\begin{aligned}\sum_{n=1}^{N}T h_n \sum_{k=1}^{n-1}h_k^{2m}M_k^{-2m}\gamma^2\Vert \partial^{m}_t u\Vert^2_{L^2(I_k)}&\leq Th_{\max}^{2m}M_{\min}^{-2m}\gamma^2\sum_{n=1}^{N} h_n \sum_{k=1}^{N}\Vert \partial^{m}_t u\Vert^2_{L^2(I_k)}\\&\leq
\gamma^2 T^2 h_{\max}^{2m} M_{\min}^{-2m}\Vert\partial^{m}_t u\Vert_{L^2(I)}^2,	\end{aligned}\]
and
\[\begin{aligned}\sum_{n=1}^{N}T h_n \sum_{k=1}^{n-1}h_k^{2m}M_k^{-2m}\Vert\psi(.,u)\Vert^2_{H^m(I_k)}&\leq Th_{\max}^{2m}M_{\min}^{-2m}\sum_{n=1}^{N} h_n \sum_{k=1}^{N}\Vert\psi(.,u)\Vert^2_{H^m(I_k)}\\&\leq
 T^2 h_{\max}^{2m} M_{\min}^{-2m}\Vert\psi(.,u)\Vert^2_{H^m(I)}.	\end{aligned}\]
Moreover, the last term can be bounded as
\[\begin{aligned}\sum_{n=1}^{N}h_n^{2m+1}M_n^{-2m}\sum_{k=1}^{n-1}\Vert \psi(.,u)\Vert^2_{H^1(I_k)}&\leq h_{\max}^{2m}M_{\min}^{-2m}\sum_{n=1}^{N} h_n \sum_{k=1}^{N}\Vert \psi(.,u)\Vert^2_{H^1(I_k)}\\&\leq
T h_{\max}^{2m} M_{\min}^{-2m}\Vert \psi(.,u)\Vert^2_{H^1(I)}.	\end{aligned}\]
Consequently, the combination of the above error bounds for Eq. \eqref{uuu} leads to the desired result.
\end{proof}
\section{Numerical results}\label{numerical experiments}
The numerical experiments are used to illustrate the efficiency of the $hp$-collocation method for the first kind Hammerstein integral equations. 
The experiments are implemented in $\textsl{Mathematica}^{\circledR}$ software platform. The programs are executed on a PC with 3.50 GHz Intel(R) Core(TM) i5-4690K processor.
In order to analyze the method, the following notations are introduced:

\[E^{N}_{1}(T)=\left(\sum\limits_{k=1}^{N}\sum\limits_{j=0}^{M_k}\frac{h_k}{2}w_{k,j}\big(u^k(x_{k,j})-u_{M_k}^k(x_{k,j})\big)^2 \right)^{\frac{1}{2}},\]
\[E^{N}_{2}(T)=\max\limits_{1\leq k \leq N}\big\vert u(t_{k})-u_{_M}^{N}(t_k)\big\vert,\]
\[E^{N}_{3}(T)=\max\limits_{t\in I}\big\vert u(t)-u_{_M}^{N}(t))\big\vert.\]

The discrete $ L^{2}$-norm error is denoted by $E^{N}_{1}(T)$, also the maximum of absolute error at the mesh knots is shown by $ E^{N}_{2}(T) $ and finally $ E^{N}_{3}(T) $ indicates  the infinite norm. Furthermore, the order of convergence $\rho_{N}$ is defined by $\log_{2}\big(\dfrac{E^{N}_{3}(T)}{E^{	{2N}}_{3}(T)}\big) $.	{
 The relation of the theoretical order of convergence stated in Theorem \ref{thm8} and  $\rho_{N}$ can be derived as 
\begin{equation}
\label{ro}
\rho_N=\log_{2}\big(\dfrac{E^{N}_{3}(T)}{E^{2N}_{3}(T)}\big)\approx \log_{2}\big(\dfrac{ch_{\max}^mM_{\min}^{-m}}{c(\dfrac{h_{\max}}{2})^mM_{\min}^{-m}}\big)=\log_2 2^m=m.
\end{equation}
One can utilize this criterion  to check the order  of convergence in practice based on the continuous injection between $L^2(\Omega)$ and $L^{\infty}(\Omega)$ \cite{brezis}.

 Let $L$ denotes the number of unknown coefficients, in this regard we have $L=\sum\limits_{n=1}^N(M_n+1)$ for the $hp$-collocation method and in a specific case, if all degrees of polynomials $  M_n $ are equal, i.e. $ M_n=M^*$, for $ n=1, \dots, N,$ then according to relation \eqref{unm}, $ L=(M^*+1)\times N. $ For convenience, we denote $ M:=M^*+1,$ and so $ L=M\times N. $} 

For the solution of the nonlinear systems which arise in the formulation of the method, one may use the Newton iteration method which needs an initial guess.
In these examples, the initial points are selected by an algorithm based on the steepest descent method.
\begin{remark}\label{rem3}
	In \cite{patterson}, an optimal control problem is solved numerically using a mesh refinement method based on collocation at Legendre-Gauss-Radau points. A relative error estimate is defined and then it is used to choose increasing the degree of polynomials or refinement of the mesh-size.  The described  scheme is called ``adaptive $hp$-collocation method". If we provide some facilities to  modify the degree of polynomials in each subinterval or change the mesh-size during the approximation procedure then the desired error could be fulfilled. 
\end{remark}
\begin{example}(\cite{ kauthenbrunner, liang2018})\label{ex9}
The following linear Volterra integral equation of first kind is considered
	\begin{equation*}
	\int_{0}^t \exp(-ts)u(s) \mathrm{d}s=\frac{\exp(-t(t+1))\sin(t)-(t+1)\cos(t)\exp(-t(t+1))+t+1}{1+(t+1)^2}, \quad t\in [0,1],
	\end{equation*}
	with the exact solution $ u(t)=\exp(-t)\cos(t).$
	
	This equation has been solved by piecewise polynomial collocation method \cite{kauthenbrunner} and a type of multi-step collocation method \cite{liang2018}. { According to Remark \ref{rem2}, we derive the numerical solution by solving linear system \eqref{eqr}.}
	 Table 1 reports the obtained error results for $hp$- and multi-step collocation methods {for diverse values of  $ N $ with  fixed degree $ M_n = M^*=3 $ and fixed step size $ h_n = h=\frac{1}{N} $ for $n = 1, \dots, N $.} By comparing the results,  we can conclude that $hp$-collocation  gives  better results. 
In addition, the best result reported in \cite{kauthenbrunner} with $ M=4$ and $ N=256 $ has the absolute error around $ 10^{-10} $ while the present scheme achieves the error $ 10^{-14} $.
Figure 1 shows the superiority of  $ hp$-version against  $ h$- and $p$-versions. The  figure on the left with fixed {$ M=M^*+1=4$  }  and different values of $ N $ depicts $ h$-version. Also, the  figure on the right demonstrates $p$-version for  each fixed $ N=1,2,4,8 $ {when $ h_n=h=\frac{1}{N}$   and various values $ M_n=M^* $ for $ n=1, \dots, N $} which can be seen as an $ hp$-version method.
 {As mentioned in the hypothesis of Theorem \ref{thm8}, $ m \leq M_{\min}+1 $, hence $ m\leq 4$. On the other hand, similar to the presented method in \cite{liang2018}, the convergence order of the present approach is $ O(h^m)$ when $h_{\max}=h$. This claim is numerically verified by relation \eqref{ro} and Table 1 in column $\rho_N.$ }
	 \end{example}
\begin{table}[ht]\label{tab1}
\begin{center}
\begin{small}
\caption{A comparison between multi-step collocation \cite{liang2018} and $hp$-collocation methods for Example {\ref{ex9}} in terms of $E_{3}^N(1)$ { for different $ N $ with fixed $ M_n = M^*=3 $ and $ h_n = h=\frac{1}{N} $ for $ n = 1, \dots, N $.}}\vspace*{0.1in}
\begin{tabular}{ccccccccccc}
\hline
\noalign{\smallskip}
$N$&& multi-step method && $ \rho_N $&& $hp$-collocation method&& $ \rho_N $\\
\noalign{\smallskip}\hline\noalign{\smallskip}
$2^1$&&$9.5385\textrm{e}-04$&&&&$2.4708\textrm{e}-05$&& \\
$2^2$&&$4.8403\textrm{e}-05$&& $ 4.3006 $&& $1.7028\textrm{e}-06$&& $ 3.8589 $\\
$2^3$ &&$2.8037\textrm{e}-06$&& $4.1096 $&&$1.1137\textrm{e}-07$&&$ 3.9345 $\\
$2^4$&&$1.6961\textrm{e}-07$&& $ 4.0470$&& $7.1140\textrm{e}-09$&&$3.9685  $\\
$2^5$&&$1.0443\textrm{e}-08$&& $ 4.0215 $&&$4.4938\textrm{e}-10$&&$ 3.9846 $\\
$2^6$&&$6.4788\textrm{e}-10$&& $4.0107$&& $2.8104\textrm{e}-11$&&$3.9990  $\\
$2^7$&&$4.2577\textrm{e}-11$&& $ 3.9276 $&& $1.7397\textrm{e}-12$&&$4.0138  $\\
\hline
&&&& $3.9276$&&&&$3.8589  $\\
\noalign{\smallskip}\hline
\end{tabular}
\end{small}
\end{center}
\end{table}
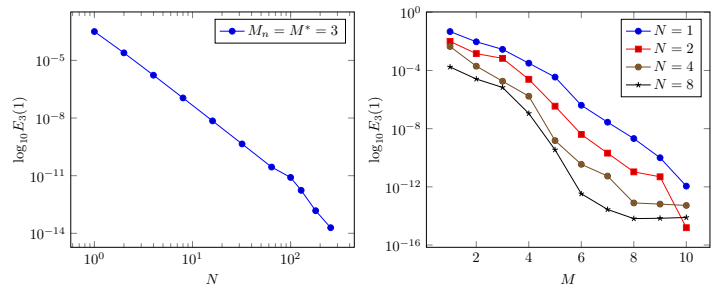
\begin{figure}[h]
\centering
\begin{tikzpicture}
[scale=0.55, transform shape]
\begin{axis}[
xlabel={$ N $},ylabel={${\log}_{10}E_3(1)$},xmode=log, ymode=log, legend entries={$ M_n=M^*=3 $}, legend pos=north east]
\addplot coordinates {
(1,3.17E-04)
(2,2.47E-05)
(4, 1.70E-06)
(8,1.11E-07)
(16,7.11E-09)
(32,4.49E-10)
(64,2.81E-11)
(100,8.10E-12)
(128,1.73E-12)
(180,1.50E-13)
(256,1.96E-14)
};
\end{axis}
\end{tikzpicture}
\begin{tikzpicture}
[scale=0.55, transform shape]
\begin{axis}[
xlabel={$M$},ylabel={${\log}_{10}E_3(1)$}, ymode=log, legend entries={$N=1$,$N=2$,$N=4$,$N=8$}, legend pos=north east]
\addplot coordinates {
(1,4.74E-02)
(2,9.39E-03)
(3,2.83E-03)
(4,3.17E-04)
(5, 3.51E-05)
(6,4.07E-07)
(7,2.78E-08)
(8,2.09E-09)
(9,1.00E-10)
(10,1.12E-12)
};
\addplot coordinates {
(1,1.00E-02)
(2,1.44E-03)
(3,6.83E-04)
(4,2.47E-05)
(5, 3.51E-07)
(6,4.04E-09)
(7,2.08E-10)
(8,1.07E-11)
(9,4.90E-12)
(10,1.55E-15)
};\addplot coordinates {
(1,4.44E-03)
(2,1.98E-04)
(3,1.83E-05)
(4,1.70E-06)
(5, 1.51E-09)
(6,3.53E-11)
(7,5.5E-12)
(8,7.81E-14)
(9,6.50E-14)
(10,5.26E-14)
};\addplot coordinates {
(1,1.74E-04)
(2,2.59E-05)
(3,6.83E-06)
(4,1.11E-07)
(5, 3.51E-10)
(6,3.34E-13)
(7,2.78E-14)
(8,6.54E-15)
(9,6.90E-15)
(10,7.83E-15)
};
\end{axis}
\end{tikzpicture}
\caption{Plots of the $ E_3^{N}(1) $ error in logarithmic scale for the   $ h$-, $p$- and $ hp$-version collocation methods for Example \ref{ex9}.}%
\label{f1}
\end{figure}
\begin{example}(\cite{ma2016, singh2016})\label{ex101}
In this example, we apply the methods to the following nonlinear Volterra integral equation of the first kind
	\begin{equation*}
	\int_{0}^t \big(\sin(t-s)+1\big)\cos(u(s))\mathrm{d}s=f(t), \quad t\in [0,1],
	\end{equation*}
	where $ f(t) $ is chosen such that $ u(t)=t$ be the exact solution. Due to the invertibility of the kernel, this equation can be converted into a second kind integral equation.  Using this idea, two numerical schemes based on Sinc-Nystr\"{o}m and Haar wavelet methods  are discussed in \cite{ma2016} and \cite{ singh2016}, respectively.
{	Table 2 and 3 show the results with comparisons.
	 Table 2 reports the comparison of Haar wavelet \cite{ma2016} and $ hp$-collocation methods with the same value of $ L $. The present scheme runs for various $ N$ with fixed step size $ h_n=h=\frac{1}{N} $, uniform mode $ M=M_n+1=2$ for $ n=1, \dots ,N$. As expected from (\ref{ro}), $ \rho_N $ is approximately equal to $ m\leq M_{\min}+1=2 $.
	 In Table 3, the best results of SE or DE Sinc-Nyst\"{o}m and $ hp$-collocation methods are provided to show the efficiency of these two algorithms. We take different $ M $ and $ N $ to achieve the best result $ u_{M}^N(t) $ with the same or near value of the column $ L $ related to Sinc-Nystr\"{o}m method. For instance, the pairs $ (1,9), ~(9,1),~ (3,3) $ have the same unknown coefficients $ L=9.$ The infinity error of these pairs are $7.90\textrm{e}-2,~  2.77\textrm{e}-12,~ 6.99\textrm{e}-08  $, respectively. Hence, the best result is for $ (9,1),$ i.e., $ u_{9}^1(t). $
	 }
	\end{example} 

	\begin{table}[ht]
		{
	\begin{center}
	\begin{small}
	\caption{The comparison of Haar wavelet \cite{singh2016} and $hp$-collocation method  for different $ N $ with fixed $ M=M_n+1=2$ and $h=h_n=\frac{1}{N}$ for $ n=1, \dots ,N$ in the sense of $E_{3}^{N}(1)$ error for Example \ref{ex101}.}\vspace*{0.1in}
	\begin{tabular}{ccccc|cccccc}
	\hline
	\noalign{\smallskip}
	$J$& $ L=2^{J+1} $ && Haar wavelet &&$  N $& $ L=MN$ && $ hp-$collocation
	\\\noalign{\smallskip}
	\hline
	\noalign{\smallskip}\hline\noalign{\smallskip}
	$2$&$8$&&$1.2\textrm{e}-03$&&$ 2^2 $&$8 $&& $ 1.62\textrm{e}-04$\\
	$3$&$16$&&$3.1\textrm{e}-04$&&$2^3$&$16 $&& $ 1.31\textrm{e}-05$\\
	$4$&$32$&&$8.0\textrm{e}-05$ &&$ 2^4 $&$32 $&& $ 2.82\textrm{e}-06$\\
	$5$ &$64$&&$2.0\textrm{e}-05$&&$ 2^5 $&$64 $&& $ 4.98\textrm{e}-07$\\
	$6$&$128$&&$5.0\textrm{e}-06$&&$ 2^6 $&$128 $&& $ 7.09\textrm{e}-08$\\
	$7$&$256$&&$1.2\textrm{e}-06$ &&$ 2^7 $&$256 $&&$8.98\textrm{e}-09  $\\
	$8$ &$512$&&$3.1\textrm{e}-07$&&$ 2^8 $&$512 $&&$1.15\textrm{e}-09  $\\
	$9$&$1028$&&$7.9\textrm{e}-08$&&$ 2^9$&$1028 $&&$ 7.19\textrm{e}-11 $\\
	\hline
	$\rho_N$&&&$2.0588$&&&&&$ 2.9632 $\\
	\noalign{\smallskip}\hline
	\end{tabular}
	\end{small}
	\end{center}
	\label{tab2}}
	\end{table}
\begin{table}[ht]	{
\begin{center}
\begin{small}
\caption{The Comparison between the best results of  Sinc-Nystr\"{o}m \cite{ma2016} and $hp$-collocation methods in the sense of $E_{3}^{N}(1)$ error for Example \ref{ex101}.}\vspace*{0.1in}
\begin{tabular}{cccc|ccccccc}
\hline
\noalign{\smallskip}
$L$&& Sinc-Nystr\"{o}m && $ (M,N) $& $ L $&& $hp$-collocation \\\noalign{\smallskip}
\hline
\noalign{\smallskip}\hline\noalign{\smallskip}
$9$&&$5.68\textrm{e}-02$&&$(9,1)$&$9$ &&$2.77\textrm{e}-12$ \\
$17$&&$3.95\textrm{e}-03$&&$(5,3)$&$15$ &&$1.37\textrm{e}-13$ \\
$33$&&$8.88\textrm{e}-05$&&$(4,8)$ &$32$&&$2.36\textrm{e}-13$ \\
$65$ &&$4.20\textrm{e}-08$&&$(4,16)$&$64$ &&$8.78\textrm{e}-13$ \\
$127$&&$8.29\textrm{e}-15$&&$(9,14)$&$126$ &&$3.83\textrm{e}-12$ \\
\noalign{\smallskip}\hline
\end{tabular}
\end{small}
\end{center}
\label{tab2}}
\end{table}

\subsubsection*{Calculation for long $T$.}
\begin{example} (\cite{sheng})\label{ex7}
In the following example, we consider solving the equation
	\begin{equation*}
	\int_{0}^t sin(t-su(s))\mathrm{d}s=f(t), \quad t\in [0,T],
	\end{equation*}
	with the exact solution $ u(t)=1$. The  Figure \ref{f7} shows  considerable results for various $T${ with fixed  $M=4$ and $ N=2 $ or fixed step size $h_n=h=\frac{1}{2} $ and $ M_n=M^*=3 $ for $ n=1,2$.}
\end{example}
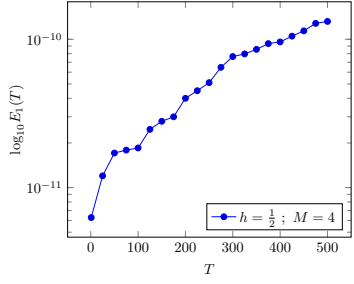
\begin{figure}[h]
	\centering
\begin{tikzpicture}
[scale=0.55, transform shape]
\begin{axis}[
xlabel={$ T $},ylabel={${\log}_{10}E_1(T)$}, ymode=log, legend entries={$h=\frac{1}{2}~;~M=4$}, legend pos=south east]
\addplot coordinates {
(1,6.28E-12)
(25,1.20E-11)
(50,1.71E-11)
(75,1.79E-11)
(100, 1.85E-11)
(125, 2.47E-11)
(150,2.80E-11)
(175, 3.00E-11)
(200,4.00E-11)
(225, 4.5E-11)
(250,5.10E-11)
(275, 6.47E-11)
(300,7.65E-11)
(325, 7.97E-11)
(350,8.55E-11)
(375, 9.35E-11)
(400,9.60E-11)
(425, 1.05E-10)
(450,1.14E-10)
(475, 1.28E-10)
(500, 1.32E-10)
};
\end{axis}
\end{tikzpicture}
\caption{Plots of the $ E_N^1(T) $ error in logarithmic scale for  different time $T$ { with fixed $ h_n=h=\frac{1}{2} $ and  $ M=M_n+1=4$ for $ n=1,2$} for Example \ref{ex7}.}%
\label{f7}%
\end{figure}
\subsubsection*{Steepest gradient solution}
\begin{example}(\cite{sheng})\label{ex5} Consider the nonlinear Volterra integral equation
	\begin{equation*}
	\int_{0}^t u^2(s) \mathrm{d}s=f(t), \quad t\in [0,10],
	\end{equation*}
	where $ f(t)=\frac{\sqrt{\pi}}{4}\big(\erf(10)+\erf(2(t-5))\big)$. The exact solution is $ u(t)=\exp({-2(t-5)^2})$. The Figure \ref{f5} depicts the results for different values of $ N $ and  $ M=M^*+1 $ { with uniform step size $ h_n=h=\frac{10}{N} $ and uniform mode $ M_n=M^* $  for all $ n=1,2$. }
\end{example}
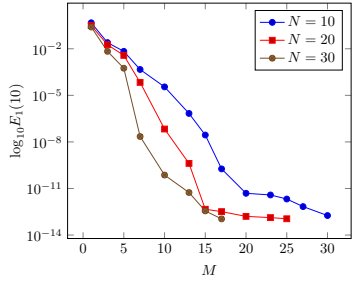
\begin{figure}[h]
	 \centering
\begin{tikzpicture}
[scale=0.55, transform shape]
\begin{axis}[
xlabel={$ M $},ylabel={${\log}_{10}E_1(10)$}, ymode=log, legend entries={$N=10$,$N=20$,$N=30$}, legend pos=north east]
\addplot coordinates {
(1,4.74E-01)
(3,2.50E-02)
(5,6.83E-03)
(7,4.63E-04)
(10, 3.51E-05)
(13,6.83E-07)
(15,2.78E-08)
(17,1.83E-10)
(20,4.90E-12)
(23,3.83E-12)
(25,2.10E-12)
(27,6.83E-13)
(30,1.81E-13)
};
\addplot coordinates {
(1,3.12E-01)
(3,1.83E-02)
(5,3.76E-03)
(7,6.83E-05)
(10, 6.78E-08)
(13,4.13E-10)
(15,4.59E-13)
(17,3.23E-13)
(20,1.60E-13)
(23,1.33E-13)
(25,1.12E-13)
};
\addplot coordinates {
(1,2.51E-01)
(3,6.83E-03)
(5,5.59E-04)
(7,2.23E-08)
(10, 7.36E-11)
(13,5.5E-12)
(15,3.59E-13)
(17,1.1E-13)
};
\end{axis}
\end{tikzpicture}
\caption{Plots of  $E^N_{1}(10)$ in logarithmic scale { for different values of $ M$ and $ N $ with fixed $ h_n=\frac{10}{N} $ and  $ M_n=M^* $ for $n=1, \dots ,N,$ } for Example \ref{ex5}.}%
\label{f5}
\end{figure}
{\subsubsection*{Unknown exact solution}
	In the following example, we consider an equation  which has a unique continuous solution in $[0,T]$  for $T<\infty$  \cite{linz}.
	
	\begin{example}\label{exe11}
		The following Volterra integral equation is considered
		\begin{equation*}
		\int_{0}^{t}  (1+t-s)^2\Big(u(s)+u^3(s)\Big)\mathrm{d}s=t^2,\quad t\in [0,T]. 
		\end{equation*}
		We know that the introduced equation has a unique solution, but the exact solution is not known. So, for the aim of  caparison, we choose $ u^5_{20}(t) $ with $ L=20\times 5=100 $ basis functions as a benchmark. Figure 4 depicts the convergence of the scheme by increasing $ M$ and $ N $ with $T=1,~ M=M_n+1=M^*+1  $ and fixed step size $h_n=h=\frac{1}{N} $ for $ n=1, \dots, N $.  Figure 5 shows the benchmark which is the approximate solution for $ T=10 $.
	\end{example}
	\begin{figure}[ht]
		\centering
		\begin{tikzpicture}
		[scale=0.55, transform shape]
		\begin{axis}[
		xlabel={$ M $},ylabel={${\log}_{10}\Vert u^5_{20}(t)-u^N_{M}(t)\Vert_\infty$}, ymode=log, legend entries={$ N=1 $,$N=3$,$N=5$}, legend pos=north east]
		\addplot coordinates {
			(2,9.37E-02)
			(4,1.65E-03)
			(6, 6.75E-04)
			(8, 7.51E-05)
			(10,1.26E-05)
			(12,5.69E-06)
			(14,9.20E-07)
			(16,3.67E-08)
			(18, 2.80E-08)
			(20, 8.98E-09)
		};
		\addplot coordinates {
			(2,1.565E-02)
			(4,3.53E-04)
			(6, 9.56E-06)
			(8, 9.5E-07)
			(10,1.98E-08)
			(12,5.24E-09)
			(14,1.25E-10)
			(16,3.45E-12)
			(18, 8.90E-13)
			(20, 1.375E-13)
		};
		\addplot coordinates {
			(2,5.67E-03)
			(4,8.77E-05)
			(6, 2.45E-06)
			(8, 8.1E-08)
			(10,1.58E-09)
			(12,5.46E-12)
			(14,8.20E-13)
			(16,1.98E-14)
			(18, 9.20E-15)
		};
		
		\end{axis}
		\end{tikzpicture}
		\caption{{ Plots of the $\Vert u^5_{20}(t)-u^N_{M}(t)\Vert_\infty$ in logarithmic scale for Example \ref{exe11}.} }%
		\label{f2}%
	\end{figure}
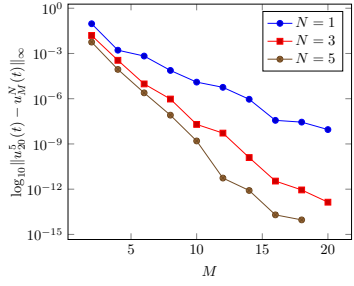
\begin{figure}[ht]
	\begin{center}
		\resizebox*{6cm}{!}{\includegraphics{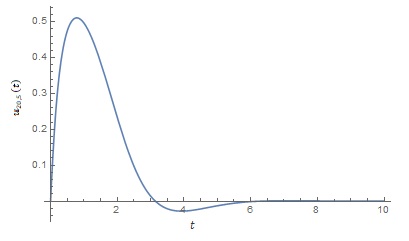}}
		\caption{{The approximate solution $ u^5_{20}(t) $ for $ T=10$. } }\label{figg2}
	\end{center}
\end{figure}
}
\subsection{Special cases}\label{it}
In this part, we present some examples of the first kind integral equations which do not fulfill the assumptions  in \cite{ egger1, kauthenbrunner, ma2016, singh2016, liang2018} and the theorems in Sections \ref{sec 2} and \ref{error analysis}. Due to the lack of enough smoothness properties for kernel $\kappa (s,t)$  and right-hand function $f(t)$, these equations could not be converted to the second kind ones. In all following examples, the advantage and efficiency of $ hp$- collocation method to approximate the non-smooth solutions  vs. $ p$- and $ h$-version methods are shown.
\subsubsection*{Non-differentiable kernel}
\begin{example}\label{ex1}
As a test problem, consider the following first kind Volterra-Hammerstein integral equation
\begin{equation*}
\int_{0}^{t} \kappa(s,t) \Big(\frac{3u^2(s)}{s+1}+\sin(u(s))\Big)\mathrm{d}s=f(t),\quad t\in[0,1],
\end{equation*}
where 
\begin{equation*}
\begin{split}
\kappa(s,t)=\left\lbrace \begin{array}{lcc}
s^2-t+5,& 0\leq t<0.5,&0\leq s\leq 1,\\
1,& 0.5\leq t\leq 1,&0\leq s\leq 1,
\end{array}\right.
\end{split}
\end{equation*}
and the exact solution is $ u(t)=t^{^3}$. Figure \ref{f1} describes the $E^N_{2}(1)$ error for $N=1,2,4$ and various $ M $  { with uniform step size $ h_n=h=\frac{1}{N} $ and uniform degree $ M_n=M^* $  for  $ n=1, \dots ,N$.}
\end{example}
\begin{figure}[h]
	\centering
\begin{tikzpicture}
[scale=0.55, transform shape]
\begin{axis}[
xlabel={$ M $},ylabel={${\log}_{10}E^N_2(1)$}, ymode=log, legend entries={$ N=1 $,$N=2$,$N=4$}, legend pos=north east]
\addplot coordinates {
(2,1.3E-01)
(3,1.32E-02)
(4, 2.86E-08)
(5,1.57E-08)
(6,1.06E-10)
(7,6.78E-12)
(8,9.05E-14)
(9,1.75E-15)
(10,1.66E-15)
};
\addplot coordinates {
(2,1.78E-02)
(3,1.74E-02)
(4, 6.59E-08)
(5,1.50E-09)
(6,1.88E-11)
(7,4.62E-13)
(8,4.99E-16)
(9,5.32E-16)
};
\addplot coordinates {
(2,4.50E-02)
(3,2.44E-03)
(4,1.06E-09)
(5,5.99E-12)
(6,2.15E-14)
(7,4.27E-15)
};
\end{axis}
\end{tikzpicture}
\caption{Plots of the $ E^N_2(1) $ error in logarithmic scale for different $ N $ and $ M $ with fixed $ h_n=\frac{1}{N} $ and  $ M_n=M^* $ for $ n=1, \dots, N$  for Example \ref{ex1}. }%
\label{f1}
\end{figure}
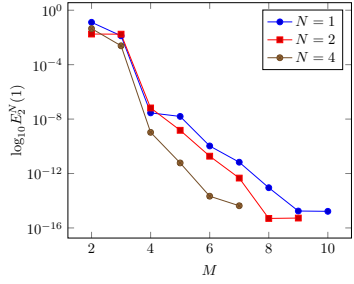
{\bf Discontinuous solution.} In this part, we focus on the nonlinear examples with discontinuous solutions.
\begin{example}\label{ex3} As another test problem consider the following integral equation
\begin{equation*}
\int_{0}^{t}  (t+4s)\Big(u(s)-\frac{s}{2}\Big)^2\mathrm{d}s=f(t), \quad t\in[0,1],
\end{equation*}
where 
\begin{equation*}
\begin{split}
u(t)=\left\lbrace \begin{array}{cc}
\frac{t-1}{2},&t<0.5,\\
2e^{-t},& t\geq 0.5,
\end{array}\right.
\end{split}
\end{equation*}
is the exact solution { and 
\begin{equation*}
 f(t)=\left\lbrace \begin{split}
& 0.75t^2,& t<0.5,\\
& -12.7174 + e^{-2t}(-4- 10t) - 0.96925 t +0.\overline{3}t^4 + 
 e^{-t} (16+ 18t + 10t^2),& t\geq 0.5.
\end{split}\right.
\end{equation*}
}
 In previous examples, we take the degree of polynomials for each $ I_k $ with $ M_k=M^*$ for all $ k=1, \dots,N $. Here, we take  $ M_k=1,~ k=1, \dots ,\frac{N}{2}$ and $ M_k={ M^*},~k=\frac{N}{2}+1, \dots ,N.$ The applicability of this scheme is verified by using less basis functions due to the behavior of the solution.
 In Figure \ref{f3}, we observe the $E^{N}_{3}(1)$ error for different values of $N$  and {$ M $ where $ M=M^*+1 $}. As expected from the theoretical achievements, by decreasing $h$ we get better numerical results. The function $f(t)$ is discontinuous and for equations with discontinuous right-hand side function or its corresponding discontinuous solution, all $ p$-version schemes are incapable to approximate the solution. On the other hand, some recent numerical methods are based on hybrid functions \cite{deh} which could be categorized into $ hp$-version methods since they approximate functions locally, but they solve the final system globally. The superiority of $ hp$-collocation method against hybrid functions method is shown in Figure \ref{fih} for various $ M $, even if we choose $ M_n=M^* $ for the whole interval. { Here, we consider fixed $ N=2$ and various $ M $ with fixed mode $ M_n=M^* $ for $ n=1,2 $. }
\end{example}
\begin{figure}[h]
	\centering
\begin{tikzpicture}
[scale=0.55, transform shape]
\begin{axis}[
xlabel={$ M$},ylabel={${\log}_{10}E_3(1)$}, ymode=log, legend entries={$h=\frac{1}{2}$,$h=\frac{1}{4}$,$h=\frac{1}{8}$,$h=\frac{1}{16}$}, legend pos=north east]
\addplot coordinates {
(2,4.74E-02)
(3,3.34E-03)
(4, 1.54E-04)
(5,5.21E-06)
(6,1.38E-07)
(7,8.10E-09)
(8,5.65E-11)
(9,9.36E-13)
};

\addplot coordinates {
(2,1.12E-02)
(3,1.89E-04)
(4, 1.49E-05)
(5,1.83E-07)
(6,6.35E-09)
(7,2.12E-11)
(8,5.63E-13)
(9,5.5E-13)
};
\addplot coordinates {
(2,8.51E-03)
(3,6.07E-05)
(4, 2.30E-06)
(5,5.79E-09)
(6,9.09E-11)
(7,2.12E-12)
(8,2.84E-13)
(9,2.42E-13)
};
\addplot coordinates {
(2,5.69E-03)
(3,7.38E-06)
(4, 2.68E-07)
(5,1.69E-09)
(6,6.21E-11)
(7,5.23E-13)
(8,1.69E-13)
};
\end{axis}
\end{tikzpicture}
\caption{Plots of the $E^{N}_{3}(1)$ error in logarithmic scale {for various $ M=M^*+1$ and $ N $ with fixed step size $ h_n=\frac{1}{N} $ and  $ M_n=1,~ n=1, \dots ,\frac{N}{2}$ and $ M_n=M^*,~n=\frac{N}{2}+1, \dots ,N$} for Example \ref{ex3}.}%
\label{f3}
\end{figure}
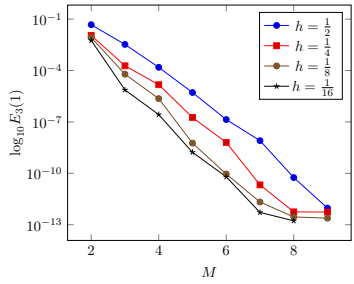

\begin{figure}[h]
	\centering
\begin{tikzpicture}
[scale=0.55, transform shape]
\begin{axis}[
xlabel={{$ M$}},ylabel={$E^{N}_{3}(1)$}, ymode=log, legend entries={$hp$-collocation, method in $\cite{deh} $ }, legend pos=north east]
\addplot coordinates {
(2,4.74E-02)
(3,3.34E-03)
(4, 1.95E-04)
(5,5.21E-06)
(6,1.38E-07)
(7,8.10E-09)
(8,5.65E-11)
(9,9.36E-13)
};
\addplot coordinates {
(2,1.11E-01)
(3,8.45E-03)
(4, 5.87E-04)
(5,4.75E-05)
(6,3.44E-6)
(7,1.83E-7)
(8,2.28E-8)
(9,4.34E-09)
};
\end{axis}
\end{tikzpicture}
\caption{Comparison between $ hp$-collocation method and the hybrid method in \cite{deh} with the $E^{N}_{3}(1)$ error in logarithmic scale with fixed $ N=2 $ and various $ M $ for Example \ref{ex3}.}%
\label{fih}
\end{figure}
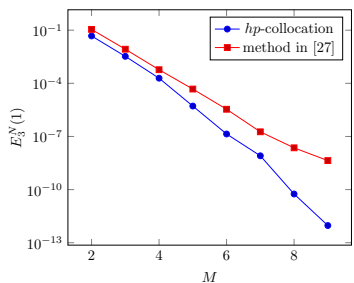
\begin{example}(\cite{sheng})\label{ex2}
	In this example, we solve the following integral equation
	\begin{equation*}
	\int_{0}^{t} u(s)\mathrm{d}s=f(t),\quad t\in[0,10],
	\end{equation*}
	where $f(t)$ can is determined such that 
	\begin{equation*}
	\begin{split}
	u(t)=\left\lbrace \begin{array}{cc}
	t,& t<5,\\
	\frac{1}{t},& t\geq 5,
	\end{array}\right.
	\end{split}
	\end{equation*}
be the exact solution.	Figure \ref{f2} displays the error $E^{N}_{1}(10)$ for different $N$ and $ M $. { In this example, we choose fixed step size $ h_n=\frac{10}{N}$ and fixed mode $ M_n=M^* $ for $ n=1,\dots,N $.} As the theoretical results  predict by increasing $M$, the error is reduced. 
\end{example}
\begin{figure}[ht]
	\centering
	\begin{tikzpicture}
	[scale=0.55, transform shape]
	\begin{axis}[
	xlabel={$ M $},ylabel={${\log}_{10}E^{N}_1(10)$}, ymode=log, legend entries={$h=\frac{10}{2}$,$h=\frac{10}{4}$,$h=\frac{10}{8}$,$h=\frac{10}{16}$}, legend pos=north east]
	\addplot coordinates {
		(4,1.27E-03)
		(6,1.83E-04)
		(8, 1.55E-06)
		(10,5.16E-08)
		(12,1.69E-09)
		(14,5.36E-11)
		(16,1.69E-12)
	};
	\addplot coordinates {
		(4,1.00E-04)
		(6,9.23E-06)
		(8, 3.29E-08)
		(10,3.75E-10)
		(12,4.26E-12)
		(14,1.05E-13)
		(16,4.84E-14)
	};
	\addplot coordinates {
		(4,3.5E-05)
		(6,3.75E-07)
		(8, 3.77E-10)
		(10,1.24E-12)
		(12,8.12E-14)
	};
	\addplot coordinates {
		(4,4.2E-06)
		(6,1.25E-08)
		(8, 3.29E-12)
		(10,3.61E-13)
	};
	\end{axis}
	\end{tikzpicture}
	\caption{Plots of the $E_1^N(10)$ error in logarithmic scale for Example \ref{ex2}. }%
	\label{f2}%
\end{figure}
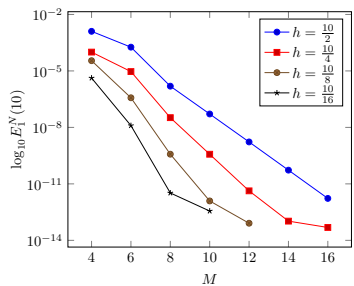
\subsubsection*{Function $\boldsymbol{\kappa(t,t)=0} $.}

\begin{example}\label{ex4} Consider the following integral equation
\begin{equation*}
\int_{0}^{t}  (s-t)e^{u(s)}\mathrm{d}s=f(t),\quad t\in[0,1],
\end{equation*}
where $ u(t)=\vert t-0.5\vert$ is the exact solution. Figure \ref{f4} shows the results in terms of different $ E_i^N(1),~i=1,2,3. $ As expected, their behavior are almost the same. The results are reported for $ N=2$ and various $ M$ { with fixed step size $ h_n=\frac{1}{2}$ and  fixed mode $ M_n=M^* $ for $ n=1,2. $} Note that the solution has finite regularity.
\end{example}.
\begin{figure}[h]
	\centering
\begin{tikzpicture}
[scale=0.55, transform shape]
\begin{axis}[
xlabel={$ M $},ylabel={$ N=2 $}, ymode=log, legend entries={${\log}_{10}E^N_{1}(1)$,${\log}_{10}E^N_{2}(1)$,${\log}_{10}E^N_{3}(1)$}, legend pos=north east]
\addplot coordinates {
(2,7.02E-04)
(3,2.44E-06)
(4, 4.91E-09)
(5,4.27E-12)
(6,8.12E-14)
};
\addplot coordinates {
(2,2.84E-04)
(3,4.55E-06)
(4, 6.74E-09)
(5,2.74E-12)
(6,1.55E-13)
};
\addplot coordinates {
(2,1.27E-03)
(3,1.83E-06)
(4, 1.55E-08)
(5,5.16E-11)
(6,1.69E-13)
};
\end{axis}
\end{tikzpicture}
\caption{Plots of the $E^N_{1}(1), E^N_{2}(1)$ and $E^N_{3}(1)$ errors in logarithmic scale for Example \ref{ex4}.}%
\label{f4}%
\end{figure}
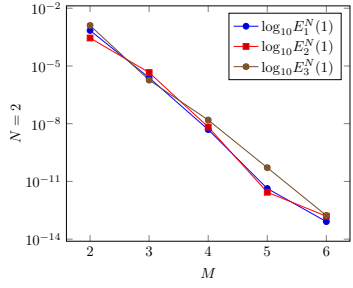
\subsubsection*{Singular solution}
As a final test problem, we consider an equation which has a weakly singular solution.
\begin{example}\label{exe10}
(\cite{sheng}) The example is the following integral equation
\begin{equation*}
\int_{0}^{t}  u^2(s)\mathrm{d}s=f(t),\quad t\in[0,1],
\end{equation*}
where $ u(t)=t^r$ with non-integer $r>\frac{1}{2}$ is the exact solution. Note that this solution has finite regularity. Figure \ref{f100} depicts the $ E_1^N(1) $ error for different $r=0.51,1.51,2.51.$ As we expect for bigger $r$,   the error is reduced significantly. { Here, we choose fixed step size $ h_n=\frac{1}{10}$ and various $ M $ with fixed mode $ M_n=M^* $ for $ n=1,\dots,10 $.}
\end{example}
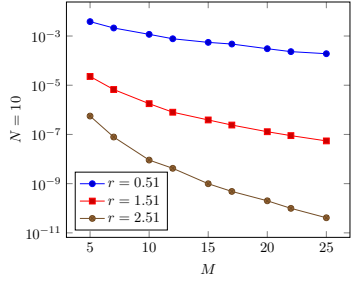
\begin{figure}[h]
	\centering
\begin{tikzpicture}
[scale=0.55, transform shape]
\begin{axis}[
xlabel={$ M $},ylabel={$ N=10 $}, ymode=log, legend entries={$r=0.51$,$ r=1.51 $,$ r=2.51 $}, legend pos=south west]
\addplot coordinates {
(5,3.85E-3)
(7, 2.12E-3)
(10,1.17E-3)
(12,7.73E-4)
(15, 5.52E-4)
(17, 4.72E-4)
(20,3.05E-4)
(22,2.32E-4)
(25,1.91E-4)
};
\addplot coordinates {
(5,2.27E-5)
(7, 6.69E-6)
(10,1.78E-6)
(12,8.01E-7)
(15, 3.88E-7)
(17, 2.41E-7)
(20,1.29E-7)
(22,8.95E-8)
(25,5.46E-8)
};
\addplot coordinates {
(5,5.58E-7)
(7, 7.85E-8)
(10,9.08E-9)
(12,4.20E-9)
(15, 9.92E-10)
(17, 4.83E-10)
(20,2.01E-10)
(22,9.92E-11)
(25,4.11E-11)
};
\end{axis}
\end{tikzpicture}
\caption{Plots of the $ E_1^N(1) $ error in logarithmic scale with  $r=0.51,1.51,2.51,$ for Example  \ref{exe10}.}
\label{f100}%
\end{figure}

\section*{Conclusion\label{SecConcl}}
Integral equations of the first kind and their approximations have a long history and many
researchers have worked on them.  In this paper, the idea of $hp$-version projection methods has been studied and a prior error analysis for the $hp$-version collocation method for the Volterra integral equations of the first kind developed. The existence and uniqueness of the solution have been investigated in the suitable Sobolev spaces under some reasonable assumptions on the nonlinearity. Numerical treatments indicate that the proposed scheme is effective and powerful to deal with smooth and non-smooth solutions, especially for long-time integration. 
{
\section*{Appendix A. Proof of the relations (\ref{eqE1}) and (\ref{eqE2})}
In this section, we need some requirements. Remind the following relation in (\ref{tau})
\[\tau=\sigma(x,\theta):=\frac{1+x}{2}\theta+\frac{1-x}{2}\]
where $x,\theta \in \Lambda$ and $ \tau\in (-1,x]. $ Let $ \mathcal{I}^\theta_{M_n}:C(\Lambda)\rightarrow \mathcal{P}_{M_n}(\Lambda)$ be the Legendre-Gauss interpolation operator. Next, it can be defined a new Legendre-Gauss interpolation operator $ \mathcal{\tilde{I}}^{\tau,x}_{M_n}:C(-1,x)\rightarrow \mathcal{P}_{M_n}(-1,x)$ with the property
\[\mathcal{\tilde{I}}^{\tau,x}_{M_n} g(\tau_{n,i})=g(\tau_{n,i}), \quad 0\leq i\leq M_n,\]
where $ \tau_{n,i}:=\tau_{n,i}(x)=\sigma(x,\theta_{n,i}) $ and $ \theta_{n,i} $ are the $ M_n+1 $ Legendre-Gauss quadrature nodes in $ \Lambda $.
Clearly,
\[\mathcal{\tilde{I}}^{\tau,x}_{M_n} g(\tau_{n,i})=g(\tau_{n,i})=g(\sigma(x,\theta_{n,i}))=\mathcal{I}^\theta_{M_n}g(\sigma(x,\theta_{n,i})), \quad\quad 0\leq i\leq M_n,\]
and by Eq. \eqref{eq2.10}, we get
\begin{equation}\label{eq4.9}
\begin{split}
\int_{-1}^{x}\mathcal{\tilde{I}}^{\tau,x}_{M_n} g(\tau)\mathrm{d}\tau&=\frac{1+x}{2} \int_{\Lambda}\mathcal{I}^\theta_{M_n}g(\sigma(x,\theta))\mathrm{d}\theta\\
&=\frac{1+x}{2} \sum_{i=0}^{M_n}g(\sigma(x,\theta_{n,i}))w_{n,i}\\
&=\frac{1+x}{2} \sum_{i=0}^{M_n}g(\tau_{n,i})w_{n,i}.
\end{split}
\end{equation}
Meanwhile,
\begin{equation}\label{eq4.10}
\int_{-1}^{x}\Big(\mathcal{\tilde{I}}^{\tau,x}_{M_n} g(\tau)\Big)^2\mathrm{d}\tau=\frac{1+x}{2} \sum_{i=0}^{M_n}g^2(\tau_{n,i})w_{n,i}.
\end{equation}
By the above argument and Lemma \ref{lem1}, we have
\begin{equation}\label{eq4.11}
\begin{split}
\int_{-1}^{x}\Big(g(\tau)-\mathcal{\tilde{I}}^{\tau,x}_{M_n} g(\tau)\Big)^2\mathrm{d}\tau&=\frac{1+x}{2} \int_{\Lambda}\Big(g(\sigma(x,\theta))-\mathcal{I}^\theta_{M_n}g(\sigma(x,\theta))\Big)^2\mathrm{d}\theta\\
&\leq cM_n^{-2m}\frac{1+x}{2}\int_{\Lambda}\Big(\partial^m_\theta g(\sigma(x,\theta))\Big)^2(1-\theta^2)^m\mathrm{d}\theta\\
&\leq cM_n\int_{-1}^x\Big(\partial^m_\tau g(\tau)\Big)^2(x-\tau)^m(1+\tau)^m\mathrm{d}\tau.
\end{split}
\end{equation}
Now, in order to estimate $ \Vert E_1\Vert$ and  $ \Vert E_2\Vert$ defined on Eq. \eqref{E12}, we rewrite them using \eqref{eq4.9} as follows
 	\begin{equation}
 	\begin{split}
 	E_1(x)
 	&=\frac{h_{n}}{2}\Big \vert  (\mathcal{I}-\mathcal{I}_{M_n}^{x})\int_{-1}^x \mathcal{\tilde{I}}_{M_n}^{\tau,x}\Big(\tilde{\kappa}^n (\tau,x){ \tilde{\psi}^{n}}(\tau,\tilde{u}^{n}_{M_n}( \tau))\Big)\mathrm{d}\tau\Big \vert,\\
 	E_2(x)&=\frac{h_{n}}{2}\Big \vert \int_{-1}^x(\mathcal{I}- \mathcal{\tilde{I}}_{M_n}^{\tau,x})\Big(\tilde{\kappa}^n (\tau,x){ \tilde{\psi}^{n}}(\tau,\tilde{u}^{n}_{M_n}( \tau))\Big)\mathrm{d}\tau\Big \vert.
 	\end{split}
 	\end{equation}
 	 In advance, we obtain the upper bound for $ \Vert E_2\Vert$. By 
Eqs. \eqref{eq4.11}, \eqref{eqt} and Cauchy-Schwarz inequality, we have 
	 \begin{equation}
	  	\begin{split}
	  	\Vert E_2(x)\Vert^2&\leq c h_n^2\int_\Lambda \int_{-1}^x\Big((\mathcal{I}-  \mathcal{\tilde{I}}_{M_n}^{\tau,x})\Big(\tilde{\kappa}^n (\tau,x){ \tilde{\psi}^{n}}(\tau,\tilde{u}^{n}_{M_n}( \tau))\Big)\Big)^2\mathrm{d}\tau \mathrm{d}x\\
	  	&\leq ch_n^{2}M_n^{-2m}\int_{\Lambda}\int_{-1}^{x}\Big(\partial_{\tau}^{m}(\tilde{\kappa}^n (\tau,x){ \tilde{\psi}^{n}}(\tau,\tilde{u}^{n}_{M_n}( \tau)))\Big)^2(x-\tau)^{m}(1+\tau)^m\mathrm{d}\tau \mathrm{d}x\\
	  	&\leq ch_n^{2m}M_n^{-2m}\int_{t_{n-1}}^{t_n}\int_{t_{n-1}}^{t}\Big(\partial_{s}^{m}(\kappa(s,t)\psi(s,u^{N}_{M}( s)))\Big)^2\mathrm{d}s\mathrm{d}t\\
	  	&\leq ch_n^{2m+1}M_n^{-2m}\sum_{i=0}^m\int_{t_{n-1}}^{t_n}\Big(\partial_{s}^{i}\psi(s,u^{N}_{M}(s)))\Big)^2\mathrm{d}s\mathrm{d}t\\
	  	&\leq ch_n^{2m+1}M_n^{-2m}\Vert \psi(.,u^{N}_{M})\Vert^2_{H^{m}(I_n)}.
	  	\end{split}
	  	\end{equation}
For the aim of deriving the upper bound for $ \Vert E_1\Vert $, by Eqs. (\ref{eq4.9}) and (\ref{eqI}) we get
\begin{equation}\label{eee}
\begin{split}
\Vert E_1(x)\Vert^2\leq &  \frac{h_n^2}{4}\int_{\Lambda}\Big[ \int_{-1}^x \mathcal{\tilde{I}}_{M_n}^{\tau,x}\Big(\tilde{\kappa}^n (\tau,x)\Big({ \tilde{\psi}^{n}}(\tau,\tilde{u}^{n}_{M_n}( \tau))-{ \tilde{\psi}^{n}}(\tau,\tilde{u}^{n}( \tau))\Big)\Big)\mathrm{d}\tau\\&
- \mathcal{I}_{M_n}^{x} \int_{-1}^x \mathcal{\tilde{I}}_{M_n}^{\tau,x}\Big(\tilde{\kappa}^n (\tau,x)\Big({ \tilde{\psi}^{n}}(\tau,\tilde{u}^{n}_{M_n}( \tau))-{ \tilde{\psi}^{n}}(\tau,\tilde{u}^{n}( \tau))\Big)\Big)\mathrm{d}\tau\\&+
(\mathcal{I}-\mathcal{I}_{M_n}^{x})\int_{-1}^{x} (\mathcal{\tilde{I}}_{M_n}^{\tau,x}-\mathcal{I})\Big(\tilde{\kappa}^n (\tau,x){ \tilde{\psi}^{n}}(\tau,\tilde{u}^{n}( \tau))\Big)\mathrm{d}\tau\\&+(\mathcal{I}-\mathcal{I}_{M_n}^{x})\int_{-1}^{x} \tilde{\kappa}^n (\tau,x){ \tilde{\psi}^{n}}(\tau,\tilde{u}^{n}( \tau))\mathrm{d}\tau\Big]^2\mathrm{d}x
\\&\leq  h^2_n\Big(E_{1,1}+E_{1,2}+E_{1,3}+E_{1,4}\Big),
\end{split}
\end{equation}
where  we introduce all terms $ E_{1,i}, i=1,2,3,4$ in sequel to obtain upper bounds for them. By Eqs. \eqref{eq4.9}, \eqref{eq4.10}, \eqref{eq4.11}, \eqref{eqt}, \eqref{lip} and H\"{o}lder inequality, we have
\begin{equation}
\begin{split}
E_{1,1}&=\int_{\Lambda}\Big[ \int_{-1}^x \mathcal{\tilde{I}}_{M_n}^{\tau,x}\Big(\tilde{\kappa}^n (\tau,x)\Big({\tilde{\psi}^{n}}(\tau,\tilde{u}^{n}_{M_n}( \tau))-{ \tilde{\psi}^{n}}(\tau,\tilde{u}^{n}( \tau))\Big)\Big)\mathrm{d}\tau\Big]^2\mathrm{d}x\\
&=\int_{\Lambda}\Big[\frac{1+x}{2} \sum_{i=0}^{M_n}\tilde{\kappa}^n (\tau_{n,i},x)({ \tilde{\psi}^{n}}(\tau_{n,i},\tilde{u}^{n}_{M_n}( \tau_{n,i}))-{ \tilde{\psi}^{n}}(\tau_{n,i},\tilde{u}^{n}( \tau_{n,i})))w_{n,i}\Big]^2\mathrm{d}x\\
&\leq c\gamma^2 \int_{\Lambda}(1+x) \sum_{i=0}^{M_n} \Big(\tilde{u}^{n}_{M_n}( \tau_{n,i})-\tilde{u}^{n}( \tau_{n,i})\Big)^2w_{n,i}\mathrm{d}x\\
& \leq c\gamma^2 \int_{\Lambda}\int_{-1}^x\Big(\mathcal{\tilde{I}}_{M_n}^{\tau,x}\tilde{u}^{n}( \tau)-\tilde{u}^{n}_{M_n}( \tau)\Big)^2\mathrm{d}\tau\mathrm{d}x
\\
&\leq c\gamma^2 \int_{\Lambda}\int_{-1}^x\Big[(\mathcal{\tilde{I}}_{M_n}^{\tau,x}\tilde{u}^{n}( \tau)-\tilde{u}^{n}( \tau))^2+(\tilde{u}^{n}( \tau)-\tilde{u}^{n}_{M_n}( \tau))^2\Big]\mathrm{d}\tau\mathrm{d}x\\
& \leq \frac{d\gamma^2}{6}\Vert e_n\Vert^2+c\gamma^2M_n^{-2m}\int_\Lambda\Big(\partial_\tau^m\tilde{u}^{n}( \tau)\Big)^2(1-\tau^2)^m\mathrm{d}\tau\\
& \leq \frac{d\gamma^2}{6}\Vert e_n\Vert^2+c\gamma^2h_n^{2m-1}M_n^{-2m}\int_{t_{n-1}}^{t_n}\Big(\partial_t^m u(t)\Big)^2\mathrm{d}t
\\&=\frac{d\gamma^2}{6}\Vert e_n\Vert^2+c\gamma^2h_n^{2m-1}M_n^{-2m}\Vert\partial_t^mu\Vert^2_{L^2(I_n)},
\end{split}
\end{equation} 
where the constant $ d $ depends on $\max\limits_{(s,t)\in I\times I}\vert \kappa(s,t)\vert.$	Next, 
\begin{equation}
\begin{split}
E_{1,2}&=\int_\Lambda\Big[\mathcal{I}_{M_n}^{x} \int_{-1}^x \mathcal{\tilde{I}}_{M_n}^{\tau,x}\Big(\tilde{\kappa}^n (\tau,x)\Big({ \tilde{\psi}^{n}}(\tau,\tilde{u}^{n}_{M_n}( \tau))-{ \tilde{\psi}^{n}}(\tau,\tilde{u}^{n}( \tau))\Big)\Big)\mathrm{d}\tau\Big]^2\mathrm{d}x\\
&=\sum_{j=0}^{M_n}\Big[ \dfrac{1+x_{n,j}}{2}\sum_{i=0}^{M_n}\tilde{\kappa}^n (\tau_{n,i},x_{n,j})\Big({ \tilde{\psi}^{n}}(\tau_{n,i},\tilde{u}^{n}_{M_n}( \tau_{n,i}))-{ \tilde{\psi}^{n}}(\tau_{n,i},\tilde{u}^{n}( \tau_{n,i}))\Big)w_{n,i} \Big]^2w_{n,j},
\end{split}
\end{equation}
where $ \tau_{n,i}=\tau_{n,i}(x_{n,j}).$ By Lipschitz condition \eqref{lip} and Eqs. \eqref{eq4.10}, \eqref{eq4.11}, \eqref{eqt} and the H\"{o}lder inequality, we have
\begin{equation}
\begin{split}
E_{1,2}&\leq c\gamma^2\sum_{j=0}^{M_n} (1+x_{n,j})w_{n,j}\sum_{i=0}^{M_n}\Big(\tilde{u}^{n}_{M_n}( \tau_{n,i})-\tilde{u}^{n}( \tau_{n,i})\Big)^2w_{n,i}\\
&\leq c\gamma^2\sum_{j=0}^{M_n} w_{n,j}\int_{-1}^{x_{n,j}}\Big(\tilde{u}^{n}_{M_n}(\tau)-\mathcal{\tilde{I}}_{M_n}^{\tau,x_{n,j}}\tilde{u}^{n}( \tau)\Big)^2\mathrm{d}\tau\\
&\leq c\gamma^2\sum_{j=0}^{M_n} w_{n,j}\int_{-1}^{x_{n,j}}\Big((\tilde{u}^{n}_{M_n}(\tau)-\tilde{u}^{n}(\tau))^2+(\tilde{u}^{n}( \tau)-\mathcal{\tilde{I}}_{M_n}^{\tau,x_{n,j}}\tilde{u}^{n}( \tau))^2\Big)\mathrm{d}\tau\\
&\leq \frac{d\gamma^2}{6} \Vert e_n \Vert^2+c\gamma^2M_n^{-2m}\int_{\Lambda}\Big(\partial_\tau^m\tilde{u}^{n}(\tau)\Big)^2(1-\tau^2)^m\mathrm{d}\tau\\
&\leq \frac{d\gamma^2}{6} \Vert e_n \Vert^2+c\gamma^2M_n^{-2m}h_n^{2m-1}\int_{t_{n-1}}^{t_n}\Big(\partial_t^mu(t)\Big)^2\mathrm{d}t\\
& = \frac{d\gamma^2}{6} \Vert e_n \Vert^2+c\gamma^2M_n^{-2m}h_n^{2m-1}\Vert\partial_t^mu\Vert^2_{L^2(I_n)}.
\end{split}
\end{equation}
 To derive an upper bound for $ E_{1,4},$ we follow Lemma \ref{lem1} and Eq. \eqref{eqt} to get
\begin{equation}
\begin{split}
E_{1,4}&=\int_{\Lambda}\Big[(\mathcal{I}-\mathcal{I}_{M_n}^{x})\int_{-1}^{x} \tilde{\kappa}^n (\tau,x){ \tilde{\psi}^{n}}(\tau,\tilde{u}^{n}( \tau))\mathrm{d}\tau\Big]^2\mathrm{d}x\\
&\leq cM_n^{-2m}\int_\Lambda \Big( \partial_{x}^m \int_{-1}^{x} \tilde{\kappa}^n (\tau,x){ \tilde{\psi}^{n}}(\tau,\tilde{u}^{n}( \tau))\mathrm{d}\tau\Big)^2(1-x^2)^m\mathrm{d}x\\
& \leq ch_n^{2m-3}M_n^{-2m}\int_{t_{n-1}}^{t_n}\Big( \partial_{t}^m \int_{t_{n-1}}^{t} \tilde{\kappa}^n (s,t){ \tilde{\psi}^{n}}(s,\tilde{u}^{n}( s))\mathrm{d}s\Big)^2\mathrm{d}t\\
& \leq ch_n^{2m-3}M_n^{-2m}\sum_{i=0}^{m-1}\int_{t_{n-1}}^{t_n}\Big( \partial_{t}^i  \psi(t,u(t))\Big)^2\mathrm{d}t\\
& \leq ch_n^{2m-3}M_n^{-2m}\Vert \psi(.,u)\Vert^2_{H^{m}(I_n)}.
\end{split}
\end{equation}
Now, it remains to estimate $ E_{1,3} $ in  \eqref{eee}. To this end, using Lemma \ref{lem1}, Eq. \eqref{eq4.9} and the Cauchy-Schwarz inequality, we deduce that
\begin{equation}
\begin{split}
\Vert E_{1,3}\Vert& \leq\int_{\Lambda}\Big[
(\mathcal{I}-\mathcal{I}_{M_n}^{x})\int_{-1}^{x} (\mathcal{\tilde{I}}_{M_n}^{\tau,x}-\mathcal{I})\Big(\tilde{\kappa}^n (\tau,x){ \tilde{\psi}^{n}}(\tau,\tilde{u}^{n}( \tau))\Big)\mathrm{d}\tau\Big]^2\mathrm{d}x\\
& =\int_{\Lambda}\Big[
(\mathcal{I}-\mathcal{I}_{M_n}^{x})\int_{\Lambda} (\mathcal{\tilde{I}}_{M_n}^{\theta}-\mathcal{I})\Big(\frac{1+x}{2}\tilde{\kappa}^n (\sigma(x,\theta),x){ \tilde{\psi}^{n}}(\sigma(x,\theta),\tilde{u}^{n}( \sigma(x,\theta)))\Big)\mathrm{d}\theta\Big]^2\mathrm{d}x.\\
\end{split}
\end{equation}
The right hand side of above equation is equal to $ D_{4,2} $ defined on Eq. (4.18) in \cite[p. 1965]{sheng}. By following the same argument of the upper bound for $ D_{4,2} $, one can conclude that
\begin{equation}\label{eee2}
\begin{split}
\Vert E_{1,3}\Vert& \leq ch_n^{2m-3}M_n^{-2m}\sum_{i=0}^{m} \int_{t_{n-1}}^{t_n} \Big(\partial_t^i \psi(t,u(t))\Big)^2\mathrm{d}t\\
&= ch_n^{2m-3}M_n^{-2m}\Vert \psi(.,u)\Vert^2_{H^m(I_n)}.
\end{split}
\end{equation}
Consequently, a combination of Eqs. \eqref{eee}-\eqref{eee2} lead to \eqref{eqE1}, i.e.,
\begin{equation}
\begin{split}
\Vert E_1\Vert^2& \leq h_n^2\Big(\frac{2d\gamma^2}{6}\Vert e_n\Vert^2+c\gamma^2M_n^{-2m}h_n^{2m-1}\Vert\partial_t^mu\Vert^2_{L^2(I_n)}+ch_n^{2m-3}M_n^{-2m}\Vert \psi(.,u)\Vert^2_{H^m(I_n)}\Big)\\
&\leq \frac{d\gamma^2h_n^2}{3}\Vert e_n\Vert^2+cM_n^{-2m}h_n^{2m-1}\Big(\gamma^2h_n^2\Vert\partial_t^mu\Vert^2_{L^2(I_n)}+\Vert \psi(.,u)\Vert^2_{H^m(I_n)}\Big).
\end{split}
\end{equation}
{\section*{Appendix B. The proof of relation (\ref{eqB}) }
\begin{equation}\label{b3}
\begin{split}
	\Big\Vert B_3(x)\Big\Vert^2&=\Big\Vert -\sum\limits_{k=1}^{n-1} \frac{h_k}{2}\Big( \tilde{\kappa}^k(. ,x), {{ \tilde{\psi}^{k}}}(.,\tilde{u}^{k}(.))\Big)+\sum\limits_{k=1}^{n-1} \frac{h_k}{2}\mathcal{I}_{M_n}^x\Big\langle \tilde{\kappa}^k(.,x), {{ \tilde{\psi}^{k}}}(.,\tilde{u}_{M_k}^{k}(.))\Big\rangle_{M_k}\Big\Vert^2
	\\
	&\leq \dfrac{1}{4}(\sum\limits_{k=1}^{n-1} h_k\Big\Vert -\int_{\Lambda} \tilde{\kappa}^k(\eta,x) {{ \tilde{\psi}^{k}}}(\eta,\tilde{u}^{k}(\eta))\mathrm{d}\eta+\mathcal{I}_{M_n}^x\Big(\int_{\Lambda} \mathcal{I}_{M_k}^\eta \Big(\tilde{\kappa}^k(\eta,x), {{ \tilde{\psi}^{k}}}(\eta,\tilde{u}_{M_k}^{k}(\eta))\Big)\mathrm{d}\eta\Big)\Big\Vert^2\\
	&\leq \dfrac{1}{4}\sum\limits_{j=1}^{n-1} h_j\sum\limits_{k=1}^{n-1} h_k\Big\Vert \int_{\Lambda} \tilde{\kappa}^k(\eta,x) {{ \tilde{\psi}^{k}}}(\eta,\tilde{u}^{k}(\eta))\mathrm{d}\eta-\mathcal{I}_{M_n}^x\Big(\int_{\Lambda} \mathcal{I}_{M_k}^\eta \Big(\tilde{\kappa}^k(\eta,x), {{ \tilde{\psi}^{k}}}(\eta,\tilde{u}_{M_k}^{k}(\eta))\Big)\mathrm{d}\eta\Big)\Big\Vert^2\\
		&\leq \dfrac{T}{4}\sum\limits_{k=1}^{n-1} h_k\Big\Vert \int_{\Lambda} \tilde{\kappa}^k(\eta,x) {{ \tilde{\psi}^{k}}}(\eta,\tilde{u}^{k}(\eta))\mathrm{d}\eta-\mathcal{I}_{M_n}^x\Big(\int_{\Lambda} \mathcal{I}_{M_k}^\eta \Big(\tilde{\kappa}^k(\eta,x), {{ \tilde{\psi}^{k}}}(\eta,\tilde{u}_{M_k}^{k}(\eta))\Big)\mathrm{d}\eta\Big)\Big\Vert^2.
\end{split}
\end{equation}
Clearly, we obtain 
\begin{equation*}
\Vert B_3(x)\Vert\leq \dfrac{3T}{4}\sum\limits_{k=1}^{n-1}h_k(B_{3,1}+B_{3,2}+B_{3,3}),
\end{equation*}
where 
\begin{equation}
\begin{split}
& B_{3,1}=\Big\Vert \int_{\Lambda} (\mathcal{I}-\mathcal{I}_{M_k}^\eta)\tilde{\kappa}^k(\eta,x) {{ \tilde{\psi}^{k}}}(\eta,\tilde{u}^{k}(\eta))\mathrm{d}\eta\Big\Vert^2\\
& B_{3,2}=\Big\Vert \int_{\Lambda} \mathcal{I}_{M_k}^\eta\Big(\tilde{\kappa}^k(\eta,x)( {{ \tilde{\psi}^{k}}}(\eta,\tilde{u}^{k}(\eta))-{{ \tilde{\psi}^{k}}}(\eta,\tilde{u}_{M_k}^{k}(\eta)))\Big)\mathrm{d}\eta\Big\Vert^2\\
& B_{3,3}=\Big\Vert (\mathcal{I}-\mathcal{I}_{M_n}^x)\int_{\Lambda}\mathcal{I}_{M_k}^\eta\Big( \tilde{\kappa}^k(\eta,x) {{ \tilde{\psi}^{k}}}(\eta,\tilde{u}_{M_k}^{k}(\eta))\Big)\mathrm{d}\eta\Big\Vert^2.
\end{split}
\end{equation}
Now using relation \eqref{eqt}, Lemma \ref{lem1}  and the Cauchy-Schwarz inequality, we have
\begin{equation}
\begin{split}
B_{3,1}&\leq c\int_{\Lambda} \int_{\Lambda} \Big((\mathcal{I}-\mathcal{I}_{M_k}^\eta)\tilde{\kappa}^k(\eta,x) {{ \tilde{\psi}^{k}}}(\eta,\tilde{u}^{k}(\eta))\Big)^2\mathrm{d}\eta\mathrm{d}x\\
&\leq cM_k^{-2m}\int_{\Lambda} \int_{\Lambda}\Big( \partial^m_{\eta}( \tilde{\kappa}^k(\eta,x) { \tilde{\psi}^{k}}(\eta,\tilde{u}^{k}(\eta)))\Big)^2(1-\eta^2)^m\mathrm{d}\eta\mathrm{d}x\\
&\leq ch_k^{2m-1}h_n^{-1}M_k^{-2m}\int_{I_n}\int_{I_k}\Big( \partial^m_{s}( \kappa(s,\tau) \psi(s,u(s)))\Big)^2\mathrm{d}s\mathrm{d}\tau\\
&\leq ch_k^{2m-1}h_n^{-1}M_k^{-2m}\sum\limits_{j=0}^{m}\int_{I_k}\Big( \partial^j_{s} \psi(s,u(s))\Big)^2\mathrm{d}s\mathrm{d}\tau\\
&= ch_k^{2m-1}M_k^{-2m}\Vert \psi(.,u(.)))\Vert^2_{H^m(I_k)}.
\end{split}
\end{equation}
In addition, by  the relations \eqref{eqI}, \eqref{eqt}, Lemma \ref{lem1},  Lipschitz condition \eqref{lip} and H\"{o}lder inequality, we get
\begin{equation}\label{b33}
\begin{split}
B_{3,2}&= \int_{\Lambda} \Big[\sum_{j=0}^{M_k}\tilde{\kappa}^k(x_{k,j},x)\Big( {{ \tilde{\psi}^{k}}}(x_{k,j},\tilde{u}^{k}(x_{k,j}))-{{ \tilde{\psi}^{k}}}(x_{k,j},\tilde{u}_{M_k}^{k}(x_{k,j}))\Big)w_{k,j}\Big]^2\mathrm{d}x\\&\leq c\gamma^2\int_\Lambda \sum_{j=0}^{M_k} \Big(\tilde{u}^{k}(x_{k,j})-\tilde{u}_{M_k}^{k}(x_{k,j})\Big)^2 w_{k,j}\mathrm{d}x\\
&\leq c\gamma^2\int_\Lambda \Big(\mathcal{I}_{M_k}^\eta \tilde{u}^{k}(\eta)-\tilde{u}_{M_k}^{k}(\eta)\Big)^2\mathrm{d}\eta\\
&\leq c\gamma^2\int_\Lambda \Big(\mathcal{I}_{M_k}^\eta \tilde{u}^{k}(\eta)-\tilde{u}^k(\eta)\Big)^2+\Big(\tilde{u}^k(\eta)-\tilde{u}_{M_k}^{k}(\eta)\Big)^2 \mathrm{d}\eta\\
&\leq c \gamma^2\Big(\Vert e_k\Vert^2+M_k^{-2m}\int_{\Lambda}(\partial_\eta^m \tilde{u}^k(\eta))^2(1-\eta^2)^m\mathrm{d}\eta\Big)\\
&\leq  c \gamma^2\Big(\Vert e_k\Vert^2+h_k^{2m-1}M_k^{-2m}\int_{I_k}(\partial_t^m u(t))^2 \mathrm{d}t\Big)
\\&\leq  c \gamma^2\Big(\Vert e_k\Vert^2+h_k^{2m-1}M_k^{-2m}\Vert\partial_t^m u(t) \Vert^2_{L^2(I_k)}\Big).
\end{split}
\end{equation}
In order to estimate $ B_{3,3} $, we use the same argument in \eqref{eee}. Hence
\[B_{3,3}\leq 3\big(B_{3,2}+B_{4,1}+B_{4,2}\big),\]
where
\begin{equation}
\begin{split}
B_{3,1}&\leq c\int_{\Lambda} \int_{\Lambda} (\mathcal{I}-\mathcal{I}_{M_k}^\eta)\Big(\tilde{\kappa}^k(\eta,x) {{ \tilde{\psi}^{k}}}(\eta,\tilde{u}^{k}(\eta))\Big)^2\mathrm{d}\eta\mathrm{d}x\\
&\leq cM_k^{-2m}\int_{\Lambda} \int_{\Lambda}\Big( \partial^m_{\eta}( \tilde{\kappa}^k(\eta,x) {{ \tilde{\psi}^{k}}}(\eta,\tilde{u}^{k}(\eta)))\Big)^2(1-\eta^2)^m\mathrm{d}\eta\mathrm{d}x\\
&\leq ch_k^{2m-1}h_n^{-1}M_k^{-2m}\int_{I_n}\int_{I_k}\Big( \partial^m_{s}( \kappa(s,\tau) \psi(s,u(s)))\Big)^2\mathrm{d}s\mathrm{d}\tau\\
&\leq ch_k^{2m-1}h_n^{-1}M_k^{-2m}\sum\limits_{j=0}^{m}\int_{I_k}\Big( \partial^j_{s} \psi(s,u(s))\Big)^2\mathrm{d}s\mathrm{d}\tau\\
&= ch_k^{2m-1}M_k^{-2m}\Vert \psi(.,u(.)))\Vert^2_{H^m(I_k)}.
\end{split}
\end{equation}
In addition, by  the relations \eqref{eqI}, \eqref{eqt}, Lemma \ref{lem1}, the Lipschitz condition \eqref{lip} and H\"{o}lder inequality, we get that
\begin{equation}
\begin{split}
B_{4,1}&= \int_{\Lambda}\Big[\mathcal{I}^x_{M_n} \int_{\Lambda}\mathcal{I}^\eta_{M_k}\Big(\tilde{\kappa}^k(\eta,x) \Big({{ \tilde{\psi}^{k}}}(\eta,\tilde{u}^{k}(\eta))-{{ \tilde{\psi}^{k}}}(\eta,\tilde{u}_{M_k}^{k}(\eta))\Big)\Big)\mathrm{d}\eta\Big]^2\mathrm{d}x\\
&=\sum_{i=0}^{M_n}\Big[\sum_{j=0}^{M_k}\tilde{\kappa}^k(x_{k,j},x_{n,i})\Big( {{ \tilde{\psi}^{k}}}(x_{k,j},\tilde{u}^{k}(x_{k,j}))-{{ \tilde{\psi}^{k}}}(x_{k,j},\tilde{u}_{M_k}^{k}(x_{k,j}))\Big)w_{k,j}\Big]^2w_{n,i},\\
\end{split}
\end{equation}
\begin{equation}
B_{4,2}=\int_{\Lambda}\Big[ (\mathcal{I}-\mathcal{I}_{M_n}^x)\int_{\Lambda}\mathcal{I}_{M_k}^\eta\Big(\tilde{\kappa}^k(\eta,x) {{ \tilde{\psi}^{k}}}(\eta,\tilde{u}^{k}(\eta))\Big)\mathrm{d}\eta\Big]^2\mathrm{d}x.
\end{equation}
Similar to \eqref{b33}, we conclude that
\begin{equation}
B_{4,1}\leq c \gamma^2\Big(\Vert e_k\Vert^2+h_k^{2m-1}M_k^{-2m}\Vert\partial_t^m \tilde{u}(t) \Vert^2_{L^2(I_k)}\Big).
\end{equation}
Furthermore, by Lemma \ref{lem1} and Cauchy-Schwarz inequality, we obtain
\begin{equation}
\begin{split}
B_{4,2}&\leq cM_n^{-2m}\int_{\Lambda}\Big[ \int_{\Lambda}\mathcal{I}^\eta_{M_k}\Big(\partial_x^m\Big(\tilde{\kappa}^k(\eta,x) {{ \tilde{\psi}^{k}}}(\eta,\tilde{u}^{k}(\eta))\Big)\Big)\mathrm{d}\eta\Big]^2(1-x^2)^m\mathrm{d}x\\
&\leq cM_n^{-2m}\int_{\Lambda}\Big[ \int_{\Lambda}\mathcal{I}^\eta_{M_k}\Big(\partial_x^m\Big(\tilde{\kappa}^k(\eta,x) {{ \tilde{\psi}^{k}}}(\eta,\tilde{u}^{k}(\eta))\Big)\Big]^2\mathrm{d}\eta\mathrm{d}x,
\end{split}
\end{equation}
moreover, by means of Lemma \ref{lem1} with $ m=1 $ and Cauchy-Schwarz inequality, we get
\begin{equation}\label{b42}
\begin{split}
B_{4,2}&\leq cM_n^{-2m}\int_{\Lambda}\Big[ \int_{\Lambda}\Big(\partial_x^m\tilde{\kappa}^k(\eta,x) {{ \tilde{\psi}^{k}}}(\eta,\tilde{u}^{k}(\eta))\Big)\Big]^2\mathrm{d}\eta\\
&~~~+M_k^{-2}\int_{\Lambda}\Big(\partial_\eta\Big(\partial_x^m\tilde{\kappa}^k(\eta,x) {{ \tilde{\psi}^{k}}}(\eta,\tilde{u}^{k}(\eta))\Big)\Big)^2(1-\eta^2)\mathrm{d}\eta\Big]\mathrm{d}x
\\
&\leq ch_k^{-1}h_{n}^{2m-1}M_n^{-2m}\int_{I_n} \int_{I_k}\Big(\partial_t^m\kappa(s,t) \psi(s,u(s))\Big)^2\mathrm{d}s\mathrm{d}t\\
&~~~+ch_kh_{n}^{2m-1}M_k^{-2}M_n^{-2m}\int_{I_n} \int_{I_k}\Big(\partial_s\Big(\partial_t^m\kappa(s,t) \psi(s,u(s))\Big)^2\mathrm{d}s\mathrm{d}t\\
& \leq ch_k^{-1}h_{n}^{2m}M_n^{-2m}\Vert \psi(.,u(.)) \Vert^2_{H^1(I_k)}.
\end{split}
\end{equation}
Consequently, the combinations of Eqs. \eqref{b3}-\eqref{b42} yields \eqref{eqB}. 
}}
\bibliographystyle{acm}
\bibliography{bibli}
\end{document}